\documentclass{article}

\usepackage[utf8]{inputenc}

\usepackage[english]{babel}
\usepackage{amsfonts}
\usepackage{amscd}
\usepackage[all]{xy}
\usepackage{graphicx}
\usepackage{csquotes}
\usepackage{enumitem}
\usepackage{mathtools}
\usepackage{overpic}
\usepackage{amsmath}
\usepackage{amsthm}
\usepackage{float}
\usepackage{amssymb}
\usepackage{textcomp}
\usepackage{arydshln}
\usepackage{mathtools}
\usepackage{accents}
\usepackage{caption}
\usepackage{phonetic}
\usepackage{tikz-cd}
\usepackage{tcolorbox}
\usepackage{graphicx}
\usepackage{arydshln}
\usepackage[linewidth=1pt]{mdframed}
\usepackage{lipsum}
\usepackage{bm}
\usepackage{multicol}
\usepackage{booktabs}
\usepackage{hyperref}
\usepackage[left=2cm,right=2cm]{geometry}

\usepackage[colorinlistoftodos, bordercolor=orange, backgroundcolor=orange!20,  linecolor=orange, textsize=scriptsize
]{todonotes}

\newtheorem{theorem}{Theorem}[section]

\newtheorem{proposition}[theorem]{Proposition}
\newtheorem{lemma}[theorem]{Lemma}
\newtheorem{corollary}[theorem]{Corollary}

\theoremstyle{definition}

\theoremstyle{remark}
\newtheorem{remark}[theorem]{Remark}

\newcommand{\R}{\mathbb{R}}
\newcommand{\T}{\mathbb{T}}

\newcommand{\C}{\mathcal{C}}

\newcommand{\G}{\mathcal{G}}

\newcommand{\K}[1]{{#1}^{[k]}}

\newcommand{\eps}{\varepsilon}
\newcommand{\domvar}{\varrho}

\title{Splitting of separatrices for rapid degenerate perturbations of the classical pendulum}
\author{Inmaculada Baldomá\thanks{Departament de Matemàtiques, Universitat Politècnica de Catalunya, Av. Diagonal, 647 08028 Barcelona.
Email: immaculada.baldoma@upc.edu.} \and Teresa M.-Seara \thanks{Departament de Matemàtiques, Universitat Politècnica de Catalunya, Av. Diagonal, 647 08028 Barcelona.
Email: tere.m-seara@upc.edu.} \and Román Moreno \thanks{Departament de Matemàtiques, Universitat Politècnica de Catalunya, Av. Diagonal, 647 08028 Barcelona.
Email: roman.moreno@upc.edu.}}

\begin{document}
\maketitle

\begin{abstract}
In this work we study the splitting distance of a rapidly perturbed pendulum $H(x,y,t)=\frac{1}{2}y^2+(\cos(x)-1)+\mu(\cos(x)-1)g\left(\frac{t}{\varepsilon}\right)$ with $g(\tau)=\sum_{|k|>1}g^{[k]}e^{ik\tau}$ a $2\pi$-periodic function and $\mu,\varepsilon \ll 1$. Systems of this kind undergo exponentially small splitting and, when $\mu\ll 1$, it is known that the Melnikov function actually gives an asymptotic expression for the splitting function provided $g^{[\pm 1]}\neq 0$. Our study focuses on the case $g^{[\pm 1]}=0$ and it is motivated by two main reasons. On the one hand the general understanding of the splitting, as current results fail for a perturbation as simple as $g(\tau)=\cos(5\tau)+\cos(4\tau)+\cos(3\tau)$. On the other hand, a study of the splitting of invariant manifolds of tori of rational frequency $p/q$ in Arnold's original model for diffusion leads to the consideration of pendulum-like Hamiltonians with
$
g(\tau)=\sin\left(p\cdot\frac{t}{\varepsilon}\right)+\cos\left(q\cdot\frac{t}{\varepsilon}\right),
$ where, for most $p, q\in\mathbb{Z}$ the perturbation satisfies $g^{[\pm 1]}\neq 0$.
\\
As expected, the Melnikov function is not a correct approximation for the splitting in this case. To tackle the problem we use a splitting formula based on the solutions of the so-called inner equation and make use of the Hamilton-Jacobi formalism. The leading exponentially small term appears at order $\mu^n$, where $n$ is an integer determined exclusively by the harmonics of the perturbation. We also provide an algorithm to compute it.
\end{abstract}

\textbf{Keywords}
Splitting of separatrices, exponentially small phenomena, Hamiltonian systems.

\textbf{MSCcodes}
37D10.

\section{Introduction}
\label{sec_introduction}

In this paper we revisit  the problem of the exponentially small splitting of separatrices for one and a half degrees of freedom Hamiltonian systems with a non-autonomous fast periodic perturbation. 
This problem has been subject of research due to the role of transversal intersections between invariant manifolds in the appearance of chaos and, when the dimension is high enough,  in instability phenomena such as Arnold diffusion. Historically, the approach to determining whether transversal intersections occur has been to provide an asymptotic expansion of the splitting distance in terms of the perturbation parameter.

The general setting is a Hamiltonian system with an analytic Hamiltonian of the form:
\[
H_0(x,y)+\mu H_1(x,y,t/\varepsilon),\qquad x,y\in\mathbb{T}\times\mathbb{R}\text{  or  } x,y\in\mathbb{R}^2,
\] 
where the unperturbed Hamiltonian, $H_0(x,y)$, has a saddle fixed point whose stable and unstable manifolds coincide along  an  homoclinic orbit, $H_1(x,y,\tau)$ is $2\pi$-periodic in the time $\tau$ and $0<\varepsilon<1$, $0\leq \mu<1$ are parameters. 
In these models, the parameter $\mu$ controls the size of the perturbation, whereas $\varepsilon$ controls its frequency. The question is to establish if the perturbed stable and unstable manifolds intersect transversely for $\varepsilon,\mu>0$.

For non-fast perturbations, that is, when $\varepsilon=1$, classical perturbation theory provides an explicit function, named
Melnikov function,  which gives the first order in $\mu$ of the splitting distance.
However, when the perturbation is fast in time, that is for $0<\varepsilon\ll 1$, the Melnikov function becomes exponentially small in $\varepsilon$ and therefore a direct application of  Melnikov theory does not lead to any conclusion unless we take the parameter $\mu$ exponentially small in $\varepsilon$. 

Since the 80's, using the seminal ideas  developed by Lazutkin (see \cite{bib_laz} for an English translation) many  works (see  \cite{bib_Tere_Amadeu}, \cite{bib_Gelfreich},  \cite{bib_exp_small} and references therein) have aimed at giving conditions for either ensuring the validity of the Melnikov prediction, or providing alternative  methods to obtain the asymptotic formula when Melnikov prediction fails to be true. 
In both cases, the asymptotic formula only describes the first order of the splitting distance if some non-degeneracy condition is met. In the so-called regular case, when the Melnikov method is valid, the condition can be explicitly given in terms of  the perturbation whereas 
in the  singular case, where Melnikov prediction fails, the non-degeneracy condition can be established by the non-vanishing of the so-called \emph{Stokes constant} $\Theta \ne 0$, which is obtained studying  a different equation, independent of the singular parameter $\varepsilon$, known as the \emph{inner equation}.

In this work we focus on the ``degenerate regular" case, that is, when the Melnikov function seems to give the asymptotic value of the splitting distance but the non-degeneracy conditions fails. This degenerated context is related with the study of the splitting of separatrices of rational tori in Arnold's original model of diffusion (\cite{bib_Arnold}) where this setting naturally appears (see Section \ref{sec_arnold_example}). 

The idea is to use the more powerful tools from the singular case, i.e., the approximation of the manifolds by the solutions of the inner equation, to overcome the difficulties added by the degeneracy. 
The novelty of our argument is the following: on the one hand, by looking at $\mu$ and $\varepsilon$ as two independent parameters, we use the analyticity of the system with respect to $\mu$ to Taylor expand the splitting distance, each of the terms carrying an exponentially small factor in $\varepsilon$; on the other hand, we find the smallest power in $\mu$ where the leading exponentially small term appears and, since it is absent in the Melnikov approximation when this power is greater than 1, we use the inner equation to prove that it is dominant. 
Our result is valid for all $\mu$ and $\varepsilon$ small enough. Furthermore, the asymptotic formula is valid for the case $\mu$ independent of $\varepsilon$ or $\mu=\mathcal{O}(\varepsilon^n)$ for any $n>0$. 

A similar example with $n=2$ was exposed in \cite{bib_Higher_order_perturbation_pendulum}. 
In that paper the authors study the splitting for the pendulum equation given by $H(x,y,t)=\frac{1}{2}y^2+\cos(x)-\varepsilon\frac{1}{2}(x+\sin(x))(\cos(2\omega t)+\cos(3\omega t))$, with $\omega$ a negative power of the perturbative parameter $\varepsilon$. They establish the non-dominance of the classical Melnikov function (which is exponentially small in $\omega$) and compute the $\varepsilon^2$ term of the Taylor expansion of the splitting function (note that this is analogous to our result, where the dominant term in the splitting is given by order $\mu^2$). However, as the authors point out, the question of the dominance of the second order in $\varepsilon$ of the splitting remains unsolved. In another paper in the same line, \cite{bib_Non_dominance_Melnikov_example}, the authors consider a Duffing equation given by $H(x,y,t)=\frac{1}{2}y^2+\frac{1}{2}x^2-\frac{1}{4}x^4+\varepsilon \frac{1}{3}x^3(\cos(2\omega t)+\cos(3\omega t))$, with $\omega$ also a negative power of $\varepsilon$. In this particular case they compute the second order in $\varepsilon$ of the splitting and show that it gives the correct asymptotic behaviour. Their proof of this dominance relies on specific computations for this model.

\subsection{Measuring the splitting distance}
Even if the method we present is quite general, we deal with a classical problem, the rapidly forced pendulum, to illustrate it. The associated  Hamiltonian will be:
\begin{equation}
H\left(x,y,\frac{t}{\varepsilon};\mu\right)=H_0(x,y)+\mu H_1\left(x,y,\frac{t}{\varepsilon}\right)
=\frac{1}{2}y^2+(\cos(x)-1)+\mu(\cos(x)-1)g\left(\frac{t}{\varepsilon}\right),
\label{eq_model}
\end{equation}
where $(x,y)\in\T\times\R$, $g(\tau)$ is a real analytic $2\pi$-periodic function with zero mean, $|\mu|\ll 1$ and $0<\varepsilon\ll 1$.
When $\mu=0$ the unperturbed system has a saddle point at $(0,0)$ with coinciding unstable and stable manifolds along a homoclinic orbit that can be parameterized as:
\begin{equation}
x=x_0(t)=4\arctan(e^t), \qquad 
y= y_0(t)=\frac{2}{\cosh(t)}, \qquad t\in \R .
\label{eq_homoclinica}
\end{equation}
When $\mu\neq 0$, $\{(0,0,\tau)\}_{\tau \in [0,2\pi]}$ is an hyperbolic periodic orbit that has stable and unstable manifolds which, in general, will not coincide. The phenomenon of the splitting of separatrices deals precisely with the study of the difference between those invariant manifolds, as shown in \ref{fig_splitting}:

\begin{figure}[H]
    \centering
    \begin{overpic}[width=0.6\textwidth]{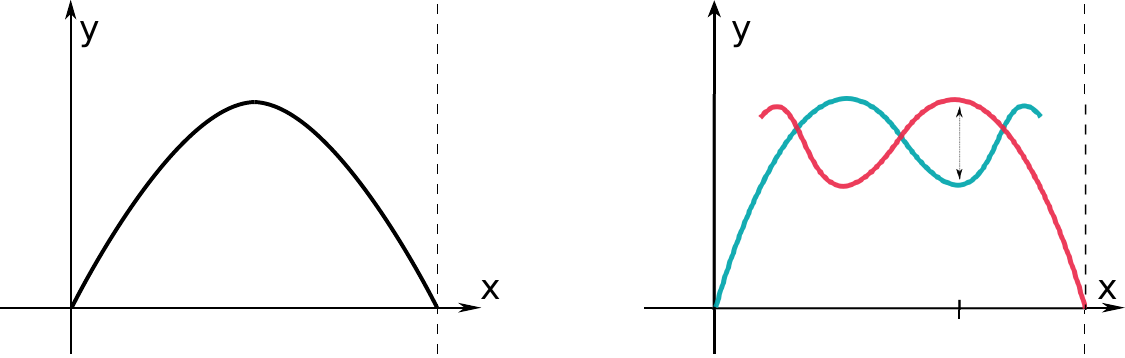}
    \put(82,6){$x_0$}
    \put(8,1){$0$}
    \put(33,1){$2\pi$}
    \put(65
    ,1){$0$}
    \put(92,1){$2\pi$}
    \put(85.5,18){$d$}
    \end{overpic}
    \caption{Left: unperturbed homoclinic. Right: distance between invariant manifolds, $d(\mathcal{W}^u,\mathcal{W}^s)$, at the point $x_0$.}
    \label{fig_splitting}
\end{figure}
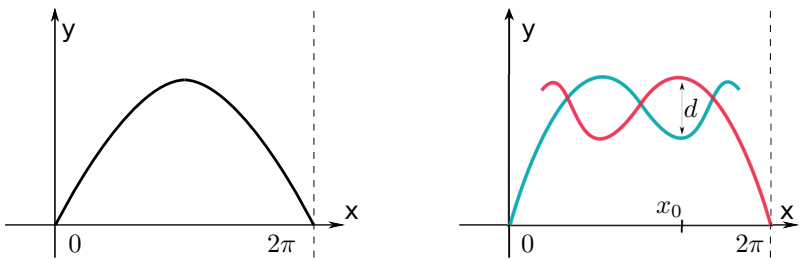

This model falls in the setting where we can apply the results of the aforementioned work, \cite{bib_exp_small}.
Let us summarize here the main ideas and specify our measure of the splitting distance as well as some standard notation.
Even when it is not essential, we profit
from the fact that the perturbed manifolds are globally expressible as a graph via a $2\pi$-periodic in $\tau$ generating function, $S(x,\tau;\mu,\varepsilon)$ (see \cite{bib_sauzin}). 
Indeed, if we denote by $W^{\mathrm{u},\mathrm{s}}(x,\tau;\mu,\varepsilon)$, where $\mathrm{u}, \mathrm{s}$ stand for unstable and stable, the graph parameterization of the perturbed manifolds, we have that
$$
W^{\mathrm{u},\mathrm{s}}(x,\tau;\mu,\varepsilon)=  \big (x,\partial_x S^{\mathrm{u},\mathrm{s}}(x,\tau;\mu,\varepsilon) \big ) 
$$
with the generating functions $S^{\mathrm{u},\mathrm{s}}$ satisfying the Hamilton-Jacobi equation
\begin{equation}\label{eq:HJnova}
H\left(x,\partial_xS,\tau;\mu\right)+\frac{1}{\eps} \partial_\tau S=0,
\end{equation}
joint with the boundary conditions
\begin{equation}\label{eq_boundarycond}
\lim _{x\to 0} \partial_x S^{\mathrm{u}}(x,\tau;\mu,\varepsilon)=0, \quad
\lim _{x\to 2\pi} \partial_x S^{\mathrm{s}}(x,\tau;\mu,\varepsilon)=0.
\end{equation}
Therefore, taking $x\in (0,2\pi)$, a measure of the splitting distance is given by
\begin{equation}\label{intro:splitdistance}
d(x,\tau;\mu,\varepsilon)=\partial_xS^{\mathrm{u}}(x,\tau;\mu,\varepsilon)-\partial_xS^{\mathrm{s}}(x,\tau;\mu,\varepsilon).
\end{equation}
Following \cite{bib_exp_small}, instead of $S^{\mathrm{u},\mathrm{s}}$ we use a different parameterization with $u$ —the time on the unperturbed homoclinic— as the parameter. 
That is, we define the new parameter $u$ by $x=x_0(u)$, 
where $x_0$ is given in \eqref{eq_homoclinica}, and we write
\begin{equation}\label{eq:tus}
\widehat{T}^{\mathrm{u},\mathrm{s}}(u,\tau;\mu,\varepsilon)=S^{\mathrm{u},\mathrm{s}}(x_0(u),\tau;\mu,\varepsilon).
\end{equation}
Then, applying the chain rule:
\begin{equation}
y=\partial_xS^{\mathrm{u},\mathrm{s}}(x,\tau;\mu,\varepsilon)=\frac{\partial u}{\partial x}\cdot\partial_uS^{\mathrm{u},\mathrm{s}}(x_0(u),\tau;\mu,\varepsilon)=\frac{\partial u}{\partial x}\cdot\partial_u \widehat{T}(u,\tau;\mu,\varepsilon)=\frac{1}{y_0(u)}\cdot\partial_u\widehat{T}(u,\tau;\mu,\varepsilon).
\label{eq_chain_rule}
\end{equation}
Note that, with this parameterization, the boundary conditions \eqref{eq_boundarycond} read
\begin{equation}
    \lim_{u\to\pm\infty}\cosh^2(u)\cdot \partial_u \widehat{T}^{\mathrm{u},\mathrm{s}}(u,\tau;\mu,\varepsilon)=0.
    \label{eq_boundarcond_T}
\end{equation}

Finally, as we expect the manifold to be close to the unperturbed homoclinic, we write
\begin{equation}\label{eq:T}
\widehat{T}^{\mathrm{u},\mathrm{s}}(u,\tau;\mu,\varepsilon)=T_0(u)+T^{\mathrm{u},\mathrm{s}}(u,\tau;\mu,\varepsilon),
\end{equation}
where $T_0(u)$ is the generating function when $\mu=0$, namely $\partial_uT_0(u)=\frac{4}{\cosh^2(u)}$. 
Summarizing, we rewrite the splitting distance in~\eqref{intro:splitdistance}, using the same notation $d$ for it as
\begin{equation}
\label{eq_formula_d}
d(u,\tau;\mu,\varepsilon)=\frac{1}{y_0(u)}\left(\partial_u\widehat{T}^\mathrm{u}(u,\tau;\mu,\varepsilon)-\partial_u\widehat{T}^\mathrm{s}(u,\tau;\mu,\varepsilon)\right)=
\frac{\cosh(u)}{2}\left(\partial_uT^\mathrm{u}(u,\tau;\mu,\varepsilon)-\partial_uT^\mathrm{s}(u,\tau;\mu,\varepsilon)\right)
\end{equation}
and therefore, analyzing the splitting distance is equivalent  to understanding the function
\begin{equation}
\Delta(u,\tau;\mu,\varepsilon):= T^{\mathrm{u}}(u,\tau;\mu,\varepsilon)-T^{\mathrm{s}}(u,\tau;\mu,\varepsilon)
\label{eq_def_Delta_outer}
\end{equation}
and its derivatives.

In the
perturbative regime $\mu\ll 1$ it is known (see \cite{bib_Scheurle}, \cite{bib_angenent_variational}, \cite{bib_transcendentally_ellison}, \cite{bib_Seara_Delshams}, \cite{bib_Separatrices_splitting_Gelfreich} and \cite{bib_exp_small}) 
that  the dominant term of the splitting distance for system \eqref{eq_model} is given by the Melnikov function, $\mathcal{M}$.
More concretely, 
if we take, for instance,  the section $x=\pi$, which corresponds to $u=0$,
the splitting distance $d$ 
is  a periodic function of $\tau$ given by
\begin{equation}
d(u=0,\tau;\mu,\varepsilon)
=\mathcal{M}(\tau;\varepsilon)\cdot \mu+\mathcal{O}\left(\frac{|\mu|^2}{\varepsilon^2}\cdot e^{-\frac{\pi}{2\varepsilon}}\right)+\mathcal{O}\left(\frac{|\mu|}{\log(1/\varepsilon)\cdot\varepsilon^2}\cdot e^{-\frac{\pi}{2\varepsilon}}\right),
\label{eq_classical_formula_splitting}
\end{equation}
where the Melnikov function $\mathcal{M}(\tau;\varepsilon)$ is given by:
\begin{align*}
\mathcal{M}(\tau;\varepsilon) &=\partial_u\left.\int_{-\infty}^\infty \frac{1}{\cosh^2(u+r)}g(\tau+r/\varepsilon)dr\right|_{u=0}=-\int_{-\infty}^\infty \frac{2\sinh(r)}{\cosh^3(r)}g(\tau+r/\varepsilon
)d
r=\\
&=-i\frac{\pi}{\varepsilon^2}\sum_{k=-\infty}^\infty g^{[k]}\cdot e^{ik\tau}\cdot\frac{k^2}{\sinh\left(\frac{k\pi}{2\cdot\varepsilon}\right)},
\end{align*}
and we have written the function $g$ as its Fourier series: 
\begin{equation}\label{eq:g}
g(\tau)=\sum_{n\in\mathbb{N}}g^{[k]}\cdot e^{ik\tau}.
\end{equation}
The harmonics of the perturbation are multiplied by increasing exponentially small factors in $\varepsilon$
\begin{align*}
\mathcal{M}(\tau;\varepsilon)&=
\frac{4\pi}{\varepsilon^2}\cdot e^{-\frac{\pi}{2\varepsilon}}\cdot \Im \left(g^{[1]}\cdot e^{i\tau}\right) +\frac{16\pi}{\varepsilon^2}\cdot e^{-\frac{\pi}{2\varepsilon}\cdot 2}\cdot\Im\left(g^{[2]}\cdot e^{2i\tau}\right)+\mathcal{O}\left(\frac{e^{-\frac{\pi}{2\varepsilon}\cdot3}}{\varepsilon^2}\right).
\end{align*}
Consequently, when $g^{[1]}\neq 0$, the asymptotic formula for the splitting~\eqref{eq_classical_formula_splitting} is:
\begin{equation}\label{eq:split1}
d(u=0,\tau;\mu,\varepsilon)=  \frac{e^{-\frac{\pi}{2\varepsilon}}}{\varepsilon^2}\cdot\left[4\pi \Im \left(g^{[1]}\cdot e^{i\tau}\right)\cdot \mu +\mathcal{O}\left(|\mu|\cdot e^{-\frac{\pi}{2\varepsilon}}\right)+\mathcal{O}(|\mu|^2)+\mathcal{O}\left(\frac{|\mu|}{\log(1/\varepsilon)}\right)\right].
\end{equation}
In this non degenerate regular case, the first term is greater than the error for $\mu,\varepsilon\to 0$ and, therefore, formula \eqref{eq:split1} gives an asymptotic formula for the splitting distance $d(u=0,\tau;\mu,\varepsilon)$. In fact, what is proved in \cite{bib_exp_small} is a more general formula, valid for any  $\mu$, including the cases where $\mu=\mathcal{O}(1)$:
\begin{equation}
d(u=0,\tau;\mu,\varepsilon)=
\frac{e^{-\frac{\pi}{2\varepsilon}}}{\varepsilon^2}\cdot \left[\Im\left(\chi^{[-1]}(\mu)\cdot e^{i\tau}\right)+\mathcal{O}\left(\frac{|\mu|}{\log(1/\varepsilon)}\right)\right], 
\label{eq_distance_1}
\end{equation}
where the Stokes constant $\chi^{[-1]}(\mu)$ is obtained through the study of some special solutions of the \emph{inner equation}, an equation independent of the parameter $\varepsilon$ which, for the pendulum system associated to Hamiltonian \eqref{eq_model}, reads:
\begin{equation}
\label{eq_inner_equation}
\partial_\tau\psi(z,\tau,\mu)+\partial_z\psi(z,\tau,\mu)=\frac{1}{8}z^2(\partial_z\psi(z,\tau,\mu))^2-2\mu \frac{g(\tau)}{z^2}.
\end{equation}
Moreover, it is proven in \cite{bib_exp_small} that, when $|\mu| \ll 1$, the Stokes constant satisfies:
\[
\chi^{[-1]}(\mu)=4\pi g^{[1]}\mu +\mathcal{O}(\mu^2)
\]
and therefore one recovers the Melnikov dominance for $\mu$ small enough and $g^{[\pm 1]}\neq 0$.

Our strategy consists in exploiting the analytic dependence of equations \eqref{eq_model} and \eqref{eq_inner_equation} on $\mu$ to prove that the error term in \eqref{eq_distance_1} is $\mathcal{O}\left(\frac{|\mu|^n}{|\log(\varepsilon)|},e^{-\frac{\pi}{2\pi\varepsilon}}\right)$. We also provide a formula for $\chi^{[-1]}(\mu)$ in terms of suitable limits of some solutions of the inner equation. From the computational point of view, dealing with the inner equation allows us to provide an effective algorithm to compute the splitting distance (see Section \ref{sec_algorithm}).

To finish, we remark that the methodology presented in this paper is independent of the particular form of the equation \eqref{eq_model} and could be applied  to any one-and-a-half-degree-of-freedom Hamiltonian system with a homoclinic orbit and a non-generic fast perturbation, performing the necessary technical changes. The paper is organized as follows: in Section \ref{sec_result_and_sketch} we present some preliminary results and state the two main Theorems, \ref{teo_main_second} and \ref{th:maintheorem}; we also give two examples of application and present an algorithm to compute the leading term of $\chi^{[-1]}(\mu)$ numerically. In Section \ref{sec_proof_theorem_main_second} we prove Theorem \ref{teo_main_second}. Finally, in Section \ref{sec_app_inner_equation} we prove Theorem \ref{th:maintheorem}. We leave some technical proofs for the appendices.

\section{Main result}
\label{sec_result_and_sketch}

\subsection{Setting and notation}

All the functions in this work depend on $u,\tau$ and $\mu$ analytically, as well as on $\varepsilon$ (not analytically). We shall write the dependence in $u,\tau,\mu$ explicitly and leave out the dependence on $\varepsilon$ unless the context requires otherwise. Notice that, as $g(\tau)$ is real analytic, there exists $\sigma_0>0$ such that $g(\tau)$ is analytic in the complex strip $\mathbb{T}_{\sigma_0}:=\{\tau\in\mathbb{C},\,\Re(\tau)\in\mathbb{T},\, |\Im(\tau)|<\sigma_0\}$ and continuous on its boundary. Since proofs typically require a finite number of reductions in the analyticity strip, when stating a result we denote by $0<\sigma<\sigma_0$ a width of analyticity for which the conclusion holds.

As for the notation, for a given $2\pi$-periodic function $g$, we denote by $G_\ell $ the sets defined as:
\begin{equation}
\label{def_Gn}
\begin{cases}
G_1=\{m\in \mathbb{Z},\,\, g^{[m]}\neq 0\},
\\
G_\ell =\left\{m\in\mathbb{Z},\,\,m=m_1+m_2+\dots+ m_\ell ,\,\,m_j\in G_1\right\}.
\end{cases}
\end{equation}
These sets will play a crucial role in our approach. The main feature we use is the following result.
\begin{lemma}
Let $g$ be a $2\pi$-periodic function. There exists $n\in \mathbb{N}$ such that $1\in G_{n}$ and $1\notin G_{\ell}$ for $\ell <n$, namely
\begin{equation}
n=n(g):=\min\{\ell \in \mathbb{N} : 1\in G_\ell\}.
\label{eq_definition_n}
\end{equation}
\end{lemma}
\begin{proof} 
We only need to prove that the set $\{\ell \in \mathbb{N}: 1 \in G_{\ell}\}$ is not empty. 
If $g$ only has one harmonic, it has to be $g^{[\pm 1]}$  (otherwise, the period would be smaller), so $n=1$. If $g$ has more than one harmonic, there exist $k_1,\dots,k_m\in G_1$ such that their greatest common divisor is $1$ (otherwise, the period would be smaller). Then, by the generalized Bézout identity there exist $\ell_1,\dots,\ell_m$ such that
$$
k_1\cdot \ell_1+\dots+k_m\cdot \ell_m=1.
$$
Notice that since $g$ is real analytic, if $k_1,\cdots, k_m \in G_1$, also $-k_1, \cdots, -k_m\in G_1$. Then one can assume that $\ell_j>0$, changing if necessary $k_j$ by $-k_j$. Hence, $\ell_*=\ell_1+\dots+\ell_m$ satisfies that $1\in G_{\ell_*}$. 
\end{proof}

\begin{remark}\label{rem:codimension}
We observe that in the space of smooth periodic functions, $\mathcal{S}$, the set 
$\mathcal{E}_0 = \{f\in \mathcal{S}: n(f)=1\}$ is generic, the set 
$\mathcal{E}_1 = \{f\in \mathcal{S}: n(f)=2\}$ has codimension one and, for $s\in\mathbb{N}$, $\mathcal{E}_{s} = \{f\in \mathcal{S}: n(f)=s+1\}$ has codimension $s$. As usual
$\mathcal{E}_s \subset \mathcal{S} \backslash (\mathcal{E}_0 \cup  \cdots \cup \mathcal{E}_{s-1})$.

Using this notation we can reformulate our aim in this paper as finding the splitting distance when the perturbation $g \in \mathcal{E}_{s}$, for some $s\geq 1$.

\end{remark}
\subsection{Main theorems}

In order to state the main results, we first summarize the relevant information about the inner equation~\eqref{eq_inner_equation} associated to the Hamiltonian~\eqref{eq_model}, which we recall is independent of the singular parameter $\varepsilon$. The results can be found in \cite{bib_Inner}.

We introduce some notation. For given $\rho,\theta>0$, let $\mathcal{D}_{\domvar,\theta}^{\mathrm{in},\pm}$ be the complex domains defined as follows
\begin{equation}\label{eq_dom_in_neg}
\mathcal{D}^{\mathrm{in},-}_{\domvar,\theta}=\{z\in\mathbb{C}; |\Im(z)|>\theta\cdot\Re(z)+\domvar\},
\qquad 
\mathcal{D}^{\mathrm{in},+}_{\domvar,\theta}=\{-z\in\mathcal{D}^{\mathrm{in},-}_{\domvar,\theta}\}
\end{equation}
(see \ref{fig_inner_domains}). For $\mu_0>0$, we introduce $B_{\mu_0}=\{\mu \in \mathbb{C} : |\mu |<\mu_0\}$ and for $\sigma>0$ we write $\T_{\sigma}=\{\tau\in\mathbb{C},\Re(\tau)\in\T,\,\, |\Im(\tau)|<\sigma\}\subset\mathbb{C}$.

\begin{figure}[H]
    \centering
    \begin{overpic}[width=\textwidth]{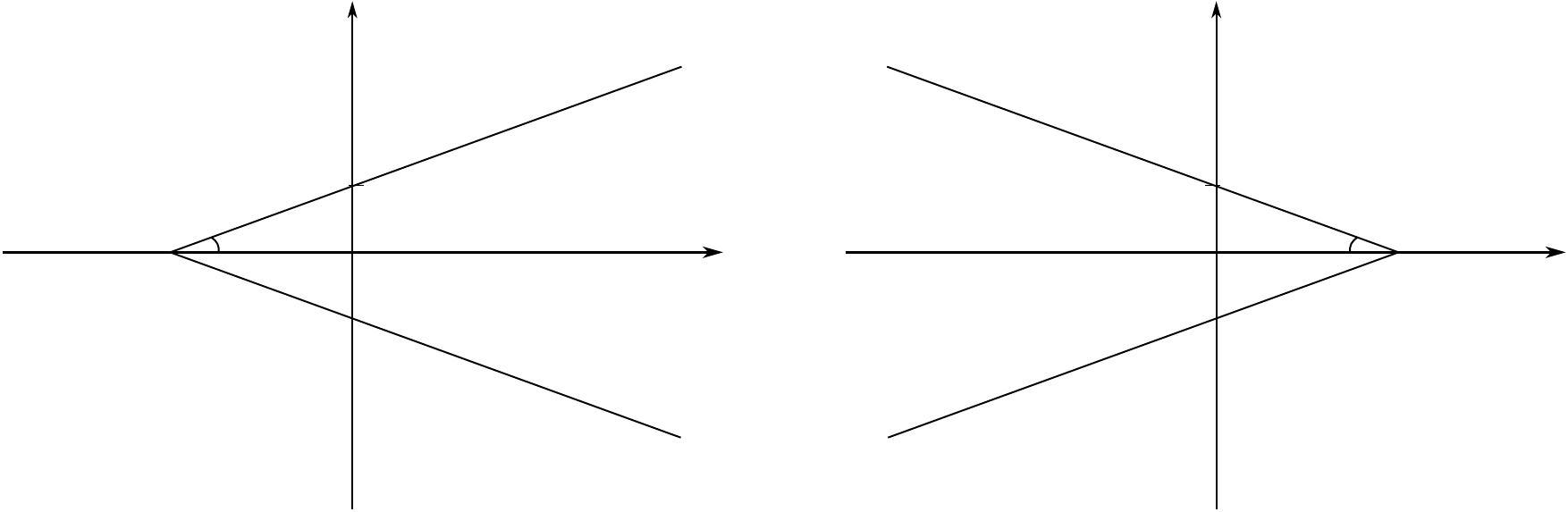}
    \put(20,22){$i\domvar$}
    \put(10,25){$\mathcal{D}^{\mathrm{in},-}_{\domvar,\theta}$}
    \put(15.5,17.2){{\footnotesize$\arctan(\theta)$}}
    \put(78,22){$i\domvar$}
    \put(90,25){$\mathcal{D}^{\mathrm{in},+}_{\domvar,\theta}$}
    \put(78,17.2){{\footnotesize$\arctan(\theta)$}}
    \end{overpic}
    \caption{Domains $\mathcal{D}^{\mathrm{in},+}_{\domvar,\theta}$ and $\mathcal{D}^{\mathrm{in},-}_{\domvar,\theta}$.}
    \label{fig_inner_domains}
\end{figure}

Now we consider the domain
\begin{equation}
\mathcal{D}^{\mathrm{in}}_{\domvar,\theta}=\mathcal{D}^{\mathrm{in},+}_{\domvar,\theta}\cap\mathcal{D}^{\mathrm{in},-}_{\domvar,\theta}\cap\{\Im(z)<0\}.
\label{eq_inner_domain}
\end{equation}
In this domain we can state the following result by paraphrasing~\cite{bib_Inner}.
\begin{theorem}[\cite{bib_Inner}]
Fix $\mu_0>0$ and $0<\arctan(\theta)<\frac{\pi}{2}$. For any periodic real analytic function $g$, there exist $\domvar_0$, $\sigma>0$ and $M=M(\domvar_0,\mu_0,\theta)$ such that $\forall \mu\in B_{\mu_0}$, $\domvar \ge \domvar_0$,  the inner equation \eqref{eq_inner_equation} has analytic solutions $\psi^{\pm}(z,\tau,\mu)$ defined in $\mathcal{D}^{\mathrm{in},\pm}_{\domvar,\theta}\times \T_\sigma\times B_{\mu_0}$, whose derivatives are uniquely determined by the condition that:
$$
|\partial_z\psi^{\pm}(z,\tau,\mu)|<M\cdot \frac{|\mu|}{|z|^3}, \qquad (z,\tau,\mu) \in D^{\mathrm{in},\pm}_{\domvar,\theta}\times \T_\sigma \times B_{\mu_0}.
$$

In addition, 
there exists an analytic function $\mathfrak{g}(z,\tau,\mu)$ defined in $\mathcal{D}_{\domvar,\theta}^{\mathrm{in}} \times \T_\sigma\times B_{\mu_0}$ satisfying  
$|\mathfrak{g}(z,\tau,\mu)|\leq M\cdot |z|^{-1}$
and such that the difference $\Delta_{\mathrm{in}}(z,\tau,\mu):=\psi^-(z,\tau,\mu)-\psi^+(z,\tau,\mu)$ is given in $\mathcal D^{\mathrm{in}}_{\domvar,\theta}\times \T_\sigma \times B_{\mu_0}$ by:
\begin{equation}
\Delta_{\mathrm{in}}(z,\tau,\mu)=\sum_{k<0}\chi^{[k]}(\mu)\cdot e^{ik(z-\tau+\mu\mathfrak{g}(z,\tau,\mu))},
\label{eq_def_Delta_inner}
\end{equation}
where $\chi^{[k]}(\mu)$ are analytic functions of $\mu$.
\label{teo_inner_existence}
\end{theorem}

Our first result, Theorem \ref{teo_main_second} below, relates the behaviour of $\chi^{[-1]}$ for $|\mu|$ small with the degree of degeneracy ($n=n(g)$) of the periodic perturbation $g$ (see Remark \ref{rem:codimension}).

\begin{theorem}
Let $g$ be a real analytic periodic function and $n=n(g)$ be defined as in~\eqref{eq_definition_n}. We consider  
$\chi^{[k]}(\mu)$ defined by~\eqref{eq_def_Delta_inner} and $\Delta^{[k]}_{\mathrm{in}}(z,\mu)$, the $k$-th coefficient in the Fourier series of $\Delta_{\mathrm{in}}$, namely
\begin{equation}
\Delta_{\mathrm{in}}(z,\tau,\mu) = \sum_{k\in \mathbb{Z}} \Delta_{\mathrm{in}}^{[k]}(z,\mu)\cdot e^{ik\tau}.
\label{def_Fourier_Deltain}
\end{equation}
Then one has:
\begin{enumerate}
    \item $\partial_\mu^j \chi^{[-1]}(0)=0$ for $j=1,\dots,n-1$ and therefore
 \begin{equation}
 \chi^{[-1]}(\mu)=  \chi^{[-1]}_n\mu^n+\mathcal{O}(\mu^{n+1} ).
 \label{eq_expansion_chi}
 \end{equation}
\item The coefficient $\chi_n^{[-1]}$, which only depends on the Fourier coefficients $\{g^{[k]}\}_{k\in \mathbb{Z}}$ with the dependence being analytic, can be computed as

\begin{equation}
\label{eq_chi_n_formula}
\chi_{n}^{[-1]} = \frac{1}{n!}\lim_{z\to -i\infty} e^{iz}  \cdot \partial_{\mu}^n \Delta_{\mathrm{in}}^{[1]}(z) .
\end{equation}

Furthermore, for the special cases $n=1,2$ we have that

$$
 \chi^{[-1]}_1=4\pi g^{[1]}, \qquad \chi^{[-1]}_2=-\frac{4\pi}{3}\sum_{k>1}\frac{g^{[k]}\cdot g^{[1-k]}}{k(1-k)},
$$

where we observe that 
$ 
\chi_2^{[-1]}= \frac{2\pi}{3}(G^2)^{[1]}$, 
with $G(\tau)$ a primitive of $g(\tau)$.
\end{enumerate}
\label{teo_main_second}
\end{theorem}

We present now the result concerning the splitting distance $\Delta(u,\tau,\mu)$ defined in~\eqref{eq_def_Delta_outer}.

\begin{theorem}\label{th:maintheorem}

Let $g$ be a real analytic periodic function, $n=n(g)$ be defined as in~\eqref{eq_definition_n} and take
$\rho>0$. 
Then there exist $\mu_0,\varepsilon_0$ such that  $\forall \mu\in(-\mu_0,\mu_0)$,   $\varepsilon\in(0,\varepsilon_0)$, $u\in (-\rho,\rho)$ and $\tau \in [0,2\pi]$, the function $\Delta$ defined in~\eqref{eq_def_Delta_outer} satisfies
\begin{align}
\partial_u\Delta(u,\tau,\mu)&= \frac{2e^{-\frac{\pi}{2\varepsilon}}}{\varepsilon^2}\cdot \left[\Im\left( \chi^{[-1]}(\mu)\cdot e^{i(\tau - u/\varepsilon)}\right)+
\mathcal{O}\left(|\mu|\cdot e^{-\frac{\pi}{2\varepsilon}}\right)+\mathcal{O}\left(\frac{|\mu|^n}{\log(1/\varepsilon)}\right)\right]
\label{eq_asymptotic_formula_theorem}
\end{align}
where $\chi^{[-1]}(\mu) =\chi_n^{[-1]}\mu^n  + O(\mu^{n+1})$ is the analytic function given in Theorem \ref{teo_main_second}.

In particular,  
\begin{enumerate}
    \item 
    If $n(g)=1$ then $g^{[1]}\neq 0$ and we have the following asymptotic formula:

 \begin{equation}
\partial_u\Delta(u,\tau,\mu)=\frac{2e^{-\frac{\pi}{2\varepsilon}}}{\varepsilon^2}\cdot\left[4\pi \Im\left(g^{[1]} \cdot e^{i\left(\tau-u/\varepsilon\right)}\right)\cdot\mu+\mathcal{O}(|\mu|^2)+\mathcal{O}\left(|\mu|\cdot e^{-\frac{\pi}{2\varepsilon}}\right)+\mathcal{O}\left(\frac{|\mu|}{\log(1/\varepsilon)}\right)\right].
\label{eq_formula_n_1}
\end{equation}
\item 
If $n(g)=2$ (and consequently $g^{[1]}=0$) we have:

\begin{equation}
\partial_u\Delta(u,\tau,\mu)=\frac{2e^{-\frac{\pi}{2\varepsilon}}}{\varepsilon^2}\cdot\left[\frac{2\pi}{3} \Im\left((G^2)^{[1]}\cdot e^{i\left(\tau-u/\varepsilon\right)}\right)\cdot \mu^2+\mathcal{O}(|\mu|^3)+\mathcal{O}\left(|\mu|\cdot e^{-\frac{\pi}{2\varepsilon}}\right)+\mathcal{O}\left(\frac{|\mu|^2}{\log(1/\varepsilon)}\right)\right],
\label{eq_formula_n_2}
\end{equation}
which is also  an asymptotic formula  when $(G^2)^{[1]}\neq 0$.
\end{enumerate}
\label{thm_main}
\end{theorem}

\begin{remark} 
In the set of functions $g$ belonging to $\mathcal{E}_{n-1}$ (see Remark \ref{rem:codimension}) we find a generic subset, namely $\{ g \in \mathcal{E}_{n-1} : \chi_{n}^{[-1]}\neq 0\}$, such that Theorem
\ref{th:maintheorem}  provides a first order asymptotic formula for the splitting distance

\begin{equation}\label{formula:asymtotic}
\partial_u\Delta(u,\tau,\mu)=\frac{2e^{-\frac{\pi}{2\varepsilon}}}{\varepsilon^2}\cdot\left[
\Im\left( \chi^{[-1]}_n \cdot e^{i(\tau-u/\varepsilon)}\right)\cdot\mu^n +
\mathcal{O}(|\mu|^{n+1} )+
\mathcal{O}\left(|\mu|\cdot e^{-\frac{\pi}{2\varepsilon}}\right)+\mathcal{O}\left(\frac{|\mu|^n}{\log(1/\varepsilon)}\right)\right]
\end{equation}
when $|\mu|^n\gg |\mu|\cdot e^{-\frac{\pi}{2\varepsilon}}$, which occurs, for instance, in the natural setting $\varepsilon>0$ small and $\mu=\mathcal{O}(\varepsilon^m)$ with $m> 0$.
\label{rm_asympt_form}
\end{remark} 

\begin{remark} 
Our result improves the formula for $\partial_u\Delta(u,\tau,\mu)$ in \cite{bib_exp_small} for equation \eqref{eq_model}, which reads

\begin{equation}
\partial_u\Delta(u,\tau,\mu)=\frac{2e^{-\frac{\pi}{2\varepsilon}}}{\varepsilon^2}\cdot \left[4\pi \Im\left(g^{[1]}\cdot e^{i\left(\tau-u/\varepsilon\right)}\right)\cdot\mu+\mathcal{O}(|\mu|^2)+\mathcal{O}\left(\frac{|\mu|}{\log(1/\varepsilon)}\right)\right].
\label{eq_original_splitting_eq}
\end{equation}
Indeed, when $g\in \mathcal{E}_0$ (see Remark \ref{rem:codimension}),  
we recover~\eqref{eq_original_splitting_eq} from Theorem \ref{th:maintheorem} (see \eqref{eq_formula_n_1}). When 
$g\in \{ g \in \mathcal{E}_{n-1} : \chi_{n}^{[-1]}\neq 0\}$, Theorem \ref{th:maintheorem} provides the asymptotic formula \eqref{formula:asymtotic}, whereas formula \eqref{eq_original_splitting_eq} only gives a non-sharp upper bound.

In addition, if $g\in \mathcal{E}_{n-1}$ but $\chi_{n}^{[-1]} =0$ (which is a non-generic codimension one phenomenon in $\mathcal{E}_{n-1}$), formula \eqref{formula:asymtotic} gives a sharper upper bound of the distance, 
\begin{equation}
|\partial_u\Delta(u,\tau,\mu)|\leq M\left|\frac{|\mu|^{n+1}}{\varepsilon^2}\cdot e^{-\frac{\pi}{2\varepsilon}} +\frac{|\mu|}{\varepsilon^{2}}\cdot e^{-\frac{\pi}{2\varepsilon}\cdot 2}+\frac{|\mu|^n}{\log(1/\varepsilon)\cdot \varepsilon^2}\cdot e^{-\frac{\pi}{2\varepsilon}}\right|,
\label{eq_main_result_upper_bound}
\end{equation}
than the one provided by formula~\eqref{eq_original_splitting_eq}.

Note that, unlike in the case $n=1$, it is possible to have a perturbation $g$ with $1\notin G_1$ and $1\in G_2$ (that is, $n(g)=2$), but $\chi_2^{[-1]}=0$. Take, for example, $g(\tau)=\cos(2\tau)+\cos(3\tau)-2\cos(4\tau)$. This function has harmonics $\pm 2$,$\pm 3$ and $\pm 4$, which means that $n(g)=2$ (see \eqref{eq_definition_n}). However, replacing $g^{[\pm 2]}=g^{[\pm 3]}=1/2$, $g^{[\pm 4]}=-1$ in the formula in Theorem \ref{teo_main_second} we see that $\chi_2^{[-1]}=0$. The study of the splitting in this extra degenerate case requires an additional analysis which is out of the scope of this paper. \end{remark}

We end this section with a corollary:
\begin{corollary}
Under the hypotheses of Theorem \ref{th:maintheorem}, if $\chi_n^{[-1]}\neq 0$, the stable and the unstable manifolds $\mathcal{W}^{\mathrm{u},\mathrm{s}}$ of the Hamiltonian system~\eqref{eq_model} intersect transversely. As a consequence, the time $2\pi \varepsilon$ 
map is conjugated to the Smale's horseshoe map of infinite symbols around any transversal homoclinic point.  
\end{corollary}

%

We prove Theorem \ref{teo_main_second} in Section \ref{sec_proof_theorem_main_second} and we use this result in the proof of Theorem \ref{thm_main}, in Section \ref{sec_app_inner_equation}.

\subsection{Examples}

In this section we provide two examples where the condition $g^{[\pm1]}\neq 0$ fails: the first one describes a very simple perturbation $g$ where we can check that $\chi^{[-1]}_1=0$ and $\chi_2^{[-1]}\neq 0$ and therefore Theorem \ref{thm_main} gives an asymptotic expression for the splitting. The second one is motivated by the study of the splitting of separatrices of rational tori in Arnold's original model of diffusion \cite{bib_Arnold}.

\subsubsection{An asymptotic formula of order \texorpdfstring{$\mathcal{O}(\mu^2\cdot e^{-\frac{\pi}{2\varepsilon}}) $}{TEXT}} 
\label{sec_example1}
Let us take the model \eqref{eq_model} with $g(\tau)=20\cos(3\tau)+16\cos(2\tau)$. In this case $g^{[\pm 1]}=0$, $g^{[\pm 3]}=10$ and $g^{[\pm 2]}=8$. By definition~\eqref{def_Gn} of $G_n$ we have that

$$
\begin{dcases}
G_1=\{-3,-2,2,3\}
\\
G_2=\{-6,-5,-4,-1,0,1,4,5,6\}.
\end{dcases}
$$
As $1\notin G_1$ but $1\in G_2$, $n=2$. Furthermore, using Theorem  \ref{teo_main_second}:
$$
\chi^{[-1]}_2=-
\frac{4\pi}{3}\frac{10\cdot 8}{3\cdot(-2)}=\frac{160}{9}\pi.
$$
We can use \eqref{eq_formula_n_2} to obtain an explicit formula for the splitting distance:
$$
\partial_u\Delta(u,\tau,\mu)=\frac{e^{-\frac{\pi}{2\varepsilon}}}{\varepsilon^2}\cdot\left[\frac{160}{9}\pi\sin(u/\varepsilon-\tau)\cdot\mu^2+\mathcal{O}(|\mu|^3)+\mathcal{O}\left(|\mu|\cdot e^{-\frac{\pi}{2\varepsilon}}\right)+\mathcal{O}\left(\frac{|\mu|^2}{\log(1/\varepsilon)}\right)\right].
$$

Note that the result in \cite{bib_exp_small} — formula \eqref{eq_original_splitting_eq}— would fail to provide an asymptotic expression for the splitting of this system, so this example corresponds to Remark \ref{rm_asympt_form}.

\subsubsection{The Arnold example}
\label{sec_arnold_example}

In \cite{bib_Arnold}, Arnold presented the following two-and-a-half-degree-of-freedom Hamiltonian system:
$$
H(\varphi_1,\varphi_2,I_1,I_2,s;\mu,\varepsilon)=\frac{1}{2}I_1^2+\frac{1}{2}I_2^2+\epsilon(\cos(\varphi_1)-1)+\epsilon\mu(\cos(\varphi_1)-1)\cdot (\sin(\varphi_2)+\cos(s)),
$$
where $(\varphi_1,\varphi_2,I_1,I_2,s)\in\mathbb{T}^2\times\mathbb{R}^2\times \T$ and $s$ is the time. 
This model had  an enormous impact on the study of instabilities of quasi-integrable Hamiltonian systems, as it is expected to display arbitrarily large drifts in the action space for arbitrarily small $\epsilon$. 
Arnold proved the existence of such instabilities under the restrictive hypothesis of exponentially smallness of $\mu$ with respect to $\epsilon$, concretely, assuming $0<\mu< e^{-\frac{\pi}{2\epsilon}}$. His approach started by showing that the invariant tori given by
$$
\mathcal{T}_{I_2^0}=\{ (\varphi_1,\varphi_2,I_1,I_2,s), \ \varphi_1=I_1=0; I_2=I_2^ 0, \ (\varphi_2,s) \in \mathbb{T}^2 \}
$$
with irrational action $I_2^0$ (also known as quasi-periodic tori) are connected through heteroclinic orbits. Observe that, for $\mu=0$, the tori only have homoclinic orbits given by $\Gamma_\omega=\{(\varphi_1,\varphi_2,I_1,I_2,s);\quad I_1^2/2+\epsilon(\cos(\varphi_1)-1)=0,\, I_2=\omega,\,(\varphi,s)\in\mathbb{T}^2\}$. His idea was that if one proved that for $\mu>0$ the stable and unstable manifolds of these tori intersect transversely along homoclinc orbits, one would also have heteroclinic orbits between nearby tori which would form a heteroclinic chain of connected tori with increasing actions $I_2$.

To establish the existence of transversal homoclinic orbits one needs an asymptotic formula for the distance between the stable and unstable manifolds. Classical perturbation theory in the parameter $\mu$ gives an exponentially small in $\epsilon$ first order and hence, in order to make his argument rigorous, Arnold needed the aforementioned condition of exponential smallness on $\mu$ with respect to $\epsilon$.

Without this hypothesis on $\mu$, proving the existence of unstable orbits is still an open question, the main difficulty being to establish the existence of transversal homoclinic/heteroclinic connections between quasi-periodic tori due to the exponentially small character of the splitting of separatrices (see \cite{bib_quasiperiodic_tori} or \cite{bib_sauzin}). An alternative way to analyze the instabilities is to study the splitting of invariant manifolds in tori with rational frequency. Let us see the problem we would face in that case. 

We focus on the invariant torus associated to a rational frequency $I_2=p/q$, and we take $(p,q)=1$:
$$
\mathcal{T}_{p/q}=\{(\varphi_1,\varphi_2,I_1,I_2,s):I_1=\varphi_1=0,I_2=p/q,(\varphi_2,s)\in\mathbb{T}^2\}.
$$
To analyze its invariant manifolds, it is convenient to perform the following change of variables and time:
$$
\begin{cases}
I_1=\sqrt{\epsilon}\cdot y\\
I_2=\frac{p}{q}+\sqrt{\epsilon}\cdot J\\
\varphi_1=x\\
\varphi_2=\varphi\\
s=\frac{t}{\sqrt{\epsilon}}
\end{cases},
$$
which shifts the invariant torus to $J=0$. In these variables the Hamiltonian becomes:
$$
\mathcal{H}(x,\varphi,y,J,t)=\frac{1}{2}y^2+\frac{p}{q\sqrt{\epsilon}}J+\frac{1}{2}J^2+(\cos(x)-1)+\mu(\cos(x)-1) \cdot \left(\sin(\varphi)+\cos\left(\frac{t}{\sqrt\epsilon}\right)\right),
$$
with equations of motion (the dot represents derivative with respect to $t$):
$$
\begin{dcases}
\dot{x}=y\\
\dot{y}=-\sin(x)-\mu\cdot\sin(x)\cdot \left(\sin(\varphi)+\cos\left(\frac{t}{\sqrt\epsilon}\right)\right)\\
\dot{\varphi}=\frac{p}{q\sqrt\epsilon}+J\\
\dot{J}=\mu\cdot(\cos(x)-1)\cdot \left(\cos(\varphi)+\cos\left(\frac{t}{\sqrt\epsilon}\right)\right)
\end{dcases}.
$$

The first two variables correspond to a pendulum with a perturbation that is fast and periodic in time but depends on the angle $\varphi$ as well. 
Even though these equations are more complex than the ones treated in this paper, it motivates our study of such degenerate systems. Indeed, since $\dot{J}=\mathcal{O}(\mu)$, assuming $J=0$,  which corresponds to the invariant torus of frequency $p/q$, in $\dot{\varphi}$, we obtain a simplified model which is a "naive first order" in $\mu$ where 
$$
\varphi(t)=\frac{p}{q\sqrt\epsilon}\cdot t
$$
and, if we restrict ourselves to the first two equations, we have:
$$
\begin{dcases}
\dot{x}=y,\\
\dot{y}=-\sin(x)-\mu\cdot\sin(x)\cdot \left(\sin\left(\frac{p}{q}\frac{t}{\sqrt\epsilon}\right)+\cos\left(\frac{t}{\sqrt\epsilon}\right)\right).
\end{dcases}
$$
We will deal with the study of splitting of resonant tori in the Arnold's model in a forthcoming paper. Here we only use this simplified model to explain our methodology.
By renaming the parameters $\varepsilon=q\sqrt\epsilon$, we obtain:
$$
\begin{dcases}
\dot{x}=y,\\
\dot{y}=-\sin(x)-\mu\cdot\sin(x)\cdot \left(\sin\left(p\cdot\frac{t}{\varepsilon}\right)+\cos\left(q\cdot\frac{t}{ \varepsilon}\right)\right).
\end{dcases}
$$
which are the equations of motion of the Hamiltonian:
$$
K(x,y,t)=\frac{1}{2}y^2+(\cos(x)-1)+\mu(\cos(x)-1)\left(\sin\left(p\cdot\frac{t}{\varepsilon}\right)+\cos\left(q\cdot\frac{t}{\varepsilon}\right)\right).
$$
We can generalize the model by adding coefficients $A$ and $B$ in the following manner:
\begin{equation}
K(x,y,t)=\frac{1}{2}y^2+(\cos(x)-1)+\mu(\cos(x)-1)\left(A\cdot\sin\left(p\cdot\frac{t}{\varepsilon}\right)+B\cdot\cos\left(q\cdot\frac{t}{\varepsilon}\right)\right).
\label{eq_derived_from_Arnolds_example}
\end{equation}
This model corresponds to \eqref{eq_model} with the function $g$ in \eqref{eq:g} given by:
$$
g\left(\tau\right)=A\cdot\sin\left(p\cdot\tau\right)+B\cdot\cos\left(q\cdot\tau\right).
$$
Note that, since $p$ and $q$ are coprime and $g$ is $2\pi$-periodic, we have that $g^{[\pm p]}=-i\frac{A}{2}$ and $g^{[\pm q]}=\frac{B}{2}$ and $g^{[k]}=0$ otherwise.
Therefore, if $p\neq 1$ and $q\neq 1$ we have that $g^{[\pm 1] } =0$. 
Using Theorem \ref{thm_main} we can state a result about the splitting of the separatrices of the hamiltonian system of Hamiltonian  \eqref{eq_derived_from_Arnolds_example}.

\begin{proposition}\label{prop:Arnolds_example}
Consider the family of systems \eqref{eq_derived_from_Arnolds_example}. 
Fix $\rho>0$, $p,q\in\mathbb{Z}$, and let $k_1^*,k_2^*\in\mathbb{Z}$ be such that 
$
k_1^*p+k_2^*q=1
$
and $|k_1^*|+|k_2^*|$ is minimal among all the integers that fulfill this condition. Let $n=|k_1^*|+|k_2^*|$. Then, there exist a constant $\Theta=\Theta(A,B,p,q)$, $\mu_0=\mu_0(A,B,p,q)>0$ and $\varepsilon_0=\varepsilon_0(\mu_0)>0$ such that for all $\mu\in(-\mu_0,\mu_0)$ and $\varepsilon\in(0,\varepsilon_0)$ the splitting distance between the unstable and stable manifolds (see \eqref{eq_formula_d} and \eqref{eq_def_Delta_outer}) for system \eqref{eq_derived_from_Arnolds_example} is given by the following formula:
\begin{align}
\partial_u\Delta(u,\tau,\mu)&=\frac{e^{-\frac{\pi}{2\varepsilon}}}{\varepsilon^2}\cdot\left[\Im\left(\Theta \cdot e^{i(\tau -u/\varepsilon)}\right)\cdot \mu^n+\mathcal{O}(\mu^{n+1})+\mathcal{O}\left(|\mu|\cdot  e^{-\frac{\pi}{2\varepsilon}}\right)+\mathcal{O}\left(\frac{|\mu|^n}{\log(1/\varepsilon)}\right)\right]
\label{eq_arnold_example_splitting}
\end{align}
for $u\in (-\rho,\rho)$ and $\tau\in [0,2\pi]$. 

Furthermore, there exists an open and dense set $U_{p,q}\subset\mathbb{R}^2$ such that if $(A,B)\in U_{p,q}$, $\Theta(A,B,p,q)\neq 0$. 

As a consequence, considering the residual set $V:=\bigcap_{(p,q)\in\mathbb{Z}^2}U_{p,q}\subset\mathbb{R}^2$, we have that if $(A,B)\in V$, then $\Theta(A,B,p,q)\neq 0$ for all $p,q$.
\end{proposition}
\begin{proof}

The proof of this proposition is straightforward by applying Theorem \ref{thm_main} to \eqref{eq_derived_from_Arnolds_example}. Indeed, fixing $p,q$ it is a consequence of the analytic dependence of the coefficient $\chi_n^{[-1]}$ with respect to $A,B$ that implies that, generically, $\chi_n^{[-1]}$ will be different from zero.  
\end{proof}

\subsection{An algorithm for computing \texorpdfstring{$\chi_n^{[-1]}$}{TEXT}}
\label{sec_algorithm}

As mentioned before, in order to have an asymptotic formula for the splitting we need $\chi_n^{[-1]}\neq 0$. Even though we cannot compute this constant analytically —except for $n=1$ and $n=2$, see \ref{teo_main_second}—, in this section we provide a numerical algorithm to check that $\chi_j^{[-1]}\neq0$ for a given $j>0$. We remark that this is an outline of a systematic algorithm rather than a rigorous numerical method, which is out of the scope of this work. In Section \ref{sec_description_numerical_method} we describe a computational algorithm to calculate solutions of a model PDE. After that, in Section \ref{sec_application_method_chi_1} we will explain the method by treating the cases $n=2$ and $n=3$, but we could extend it to any $n$ by deriving the corresponding equation. We also present a concrete computation of $\chi_2^{[-1]}$ and $\chi_3^{[-1]}$. In the case of $\chi_2^{[-1]}$ we compare the numerical result with the theoretical result of the example presented in Section \ref{sec_example1}.

\subsubsection{A model PDE}
\label{sec_description_numerical_method}
Let $h(z,\tau)$ be an analytic function $2\pi$-periodic in $\tau$ having a finite number of Fourier coefficients and an asymptotic formal expansion as $z\to+\infty$.
$$
h(z,\tau)=\sum_{\ell\geq 1}\frac{1}{z^\ell}\cdot h_\ell(\tau)=\sum_{k=-M}^Mh^{[k]}(z)\cdot e^{ik\tau}.
$$
We are interested in solutions of
$$
\partial_\tau f(z,\tau)+\partial_zf(z,\tau)=h(z,\tau)
$$
that are $2\pi$-periodic in $\tau$ with boundary conditions $\lim_{\Re(z)\to\pm\infty}f(z,\tau)=0$. We use the following method.
\paragraph{Step 1} Write the ODE for the Fourier coefficients:
\begin{equation}
\label{eq_fourier_model}
ikf^{[k]}(z)+\frac{d}{dz}f^{[k]}(z)=h^{[k]}(z).
\end{equation}
\paragraph{Step 2}Truncate the formal expansion of $h^{[k]}(z)$ up to some order $N$,
$$
h^{[k]}(z)=\sum_{\ell=1}^N\frac{1}{z^\ell}\cdot h_\ell^{[k]},
$$
and solve \eqref{eq_fourier_model} by equating terms of the same order. This provides an approximated solution of the Fourier coefficients of $f$:
$$
f^{[k]}(z)\approx \sum_{\ell=1}^N \frac{1}{z^{\ell}}\cdot f_{\ell}^{[k]}:=\tilde{f}^{[k]}(z)
$$
when $\Re(z)\gg 1$ if the boundary condition $\lim_{\Re(z)\to+\infty}f(z,\tau)=0$ is considered or when $\Re(z)\ll -1$ otherwise.
\paragraph{Step 3} For any $z_f$, select $z_0$ with $|\Re(z_0)|\gg 1$ and $\Im(z_0)=\Im(z_f)$. Set
$$
\varphi(t)=f^{[k]}(z_0+t),
$$
which is a solution of
$$
\varphi'(t)=-ik\varphi(t)+h^{[k]}(z_0+t),
$$
and numerically integrate from $t_0=0$ to $t_f=\Re(z_f)-\Re(z_0)$ with initial condition $\varphi(0)=\tilde{f}^{[k]}(z_0)$. Then, $f^{[k]}(z_f)\approx \varphi(t_f)$.

\subsubsection{The algorithm to compute \texorpdfstring{$\chi_n^{[-1]}$}{TEXT}}
\label{sec_application_method_chi_1}

We first expand the solutions of the inner equation \eqref{eq_inner_equation}

$$
\partial_\tau\psi(z,\tau,\mu)+\partial_z\psi(z,\tau,\mu)=\frac{1}{8}z^2(\partial_z\psi(z,\tau,\mu))^2-2\mu \frac{g(\tau)}{z^2}
$$
in power series in $\mu$ and in Fourier series (see \ref{section_conditions_for_delta_0} for more details):
$$\psi^{\pm}(z,\tau,\mu)=\sum_{j\geq 1}\psi_j^\pm(z,\tau)\cdot\mu^j
,\qquad \psi_j^{\pm}(z,\tau)=\sum_{k\in\mathbb{Z}}\psi_j^{\pm,[k]}(z)\cdot e^{iz\tau}.$$
Using formula \eqref{eq_chi_n_formula}, since $\frac{1}{n!}\partial_\mu^n\Delta_{\mathrm{in}}^{[1]}(z)=\psi_n^{-,[1]}(z)-\psi_n^{+,[1]}(z)$, we only need to numerically compute $\psi_{n}^{\pm,[1]}(-i\rho)$ for $\rho>0$ large enough and approximate the limit:
\begin{equation}
\label{eq_approx_chi_1}
\chi_n^{[-1]}=\lim_{z\to-i\infty}e^{iz}\cdot\Delta_{\mathrm{in},n}^{[1]}(z)\approx e^{\rho}\cdot(\psi_n^{-,[1]}(-i\rho)-\psi_n^{+,[1]}(-i\rho)).
\end{equation}

The steps are the following:
\paragraph{Step 1} Let $\sigma_0>0$ be such that $g$ is analytic in the complex strip $\mathbb{T}_{\sigma_0}$. Fix an accuracy $\delta>0$ and take $M$ big enough such that 
$$
||g||_{\sigma_0}:=\max_{\tau\in\mathbb{T}_{\sigma_0}}|g(\tau)|\leq \frac{\delta}{2} \cdot e^{M\sigma_0/2}\cdot(1-e^{\sigma_0/2}).
$$
Define $\tilde{g}(\tau)=\sum_{|k|< M}g^{[k]}e^{ik\tau}$. Since $|g^{[k]}|\leq ||g||_{\sigma_0}e^{-|k|\sigma_0}$, we have that $\max_{\tau\in\mathbb{T}_{\sigma_0/2}}|g(\tau)-\tilde{g}(\tau)|\leq \delta$. Consider the approximated inner equation

$$
\partial_\tau\psi(z,\tau,\mu)+\partial_z\psi(z,\tau,\mu)=\frac{1}{8}z^2(\partial_z\psi(z,\tau,\mu))^2-2\mu\frac{\tilde{g}(\tau)}{z^2}.
$$
\paragraph{Step 2} The functions $\partial_z\psi_1^{\pm}$ satisfy $\lim_{|\Re(z)|\to\infty}\partial_z\psi_1^{\pm}=0$ and the equation 
$$
\partial_\tau(\partial_z\psi_1^{\pm})+\partial_z(\partial_z\psi_1^{\pm})=4\mu\frac{\tilde{g}(\tau)}{z^3}.
$$
Use the method in \ref{sec_description_numerical_method} to compute $\partial_z\psi_1^{\pm,[k]}(z)$ for $\Im(z)\leq -\rho$ with $|k|\leq M$ and $k\in G_1$.
\paragraph{Step 3}The functions $\partial_z\psi_j^{\pm}$ satisfy $\lim_{|\Re(z)|\to\infty}\partial_z\psi_j^{\pm}=0$ and 
$$
\partial_\tau(\partial_z\psi_j^{\pm})+\partial_z(\partial_z\psi_j^{\pm})=h_j^{\pm},
$$
with $h_j^{\pm}$ depending on $\partial_z\psi_1^\pm$, $\dots$, $\partial_z\psi_{j-1}^{\pm}$ (already computed). Solve it using the method in the previous section only taking into account the Fourier coefficients indexed by $k\in G_j$, $|k|\leq M$.

\paragraph{Step 4}For $\psi_n^\pm$ we only need to take into account its first Fourier coefficients, $\psi_n^{\pm,[1]}$, which satisfy

$$
i\psi_n^{\pm,[1]}(z)+\partial_z\psi_n^{\pm,[1]}(z)=\frac{1}{8}z^2\sum_{l=1}^{n-1}\sum_{k=-M}^M\partial_z\psi_l^{\pm,[1-k]}(z)\cdot\partial_z\psi_{n-l}^{\pm,[k]}(z)
$$
and, again, apply the method in Section \ref{sec_description_numerical_method}, for $z_f=-i\rho$.
\paragraph{Step 5} Finally compute $\chi_n^{[-1]}$ using the approximation in \eqref{eq_approx_chi_1}.

\subsubsection{Explicit computation for \texorpdfstring{$n=2$}{TEXT} and \texorpdfstring{$n=3$}{TEXT}}
\label{sec_explicit_computation_n_3_2}
We write here explicitly the equations for $\psi_{2}^{\pm,[k]}(z)$ and $\psi_{3}^{\pm,[k]}(z)$:

\begin{equation}
\label{eq_inner_order_2_fourier}
ik\psi_2^{\pm,[k]}(z)+\partial_z\psi_2^{\pm,[k]}(z)=\frac{1}{8}z^2\sum_{m\in\mathbb{Z}}\partial_z\psi_1^{\pm,[m]}(z)\cdot\partial_z\psi_1^{\pm,[k-m]}(z)
\end{equation}
for $j=2$ and

\begin{equation}
\label{eq_inner_order_3_fourier}
ik\psi_{3}^{\pm,[k]}(z)+\partial_z\psi_{3}^{\pm,[k]}(z)=\frac{1}{4}z^2\sum_{m\in\mathbb{Z}}\partial_z\psi_{1}^{\pm,[m]}(z)\cdot\partial_z\psi_{2}^{\pm,[k-m]}(z)
\end{equation}
for $j=3$. Since on the right-hand sides of \eqref{eq_inner_order_2_fourier} and \eqref{eq_inner_order_3_fourier} the functions $\partial_z\psi_1^{[\pm]}$ and $\partial_z\psi_2^{[\pm]}$ appear, we also need their respective equations:

\begin{equation}
\label{eq_inner_order_1_derivative_fourier}
ik\partial_z\psi_1^{\pm,[k]}(z)+\partial_z(\partial_z\psi_1^{\pm,[k]})(z)=\frac{4}{z^3}g^{[k]}
\end{equation}
for $\partial_z\psi_1^{\pm,[k]}(z)$ and

\begin{equation}
    \label{eq_inner_order_2_derivative}
    ik\partial_z\psi_2^{\pm,[k]}(z)+\partial_z(\partial_z\psi_2^{\pm,[k]})(z)=\frac{1}{4}z\sum_{m\in\mathbb{Z}}\partial_z\psi_1^{\pm,[m]}(z)\partial_z\psi_1^{\pm,[k-m]}(z)+\frac{1}{4}z^2\sum_{m\in\mathbb{Z}}\partial_z\psi_1^{\pm,[m]}(z)\partial_z^2\psi_1^{\pm,[k-m]}(z)
\end{equation}
for $\partial_{z}\psi_2^{\pm,[k]}$. Finally, as $\partial_z^2\psi_1^{\pm,[k]}$ appears on the right-hand side of \eqref{eq_inner_order_2_derivative}, we need the corresponding equation, namely
\begin{equation}
\label{eq_inner_order_1_second_derivative_fourier}
    ik\partial_z^2\psi_1^{\pm,[k]}(z)+\partial_z(\partial_z^2\psi_1^{\pm,[k]}(z))=-\frac{12}{z^4}g^{[k]}.
\end{equation}

As an example of computation of $\chi_2^{[-1]}$ we take the perturbation $g(\tau)=20\cos(3\tau)+16\cos(2\tau)$, already discussed in Section \ref{sec_example1}. On Table \ref{tbl_numerics} and the left panel of Figure \ref{fig_numerics} we see the values we obtained using different values for $\rho$ with $\Re(z_0)=40$ and $N=20$. We see that from $\rho=4$ to $\rho=10$ the numerical value coincides precisely with the theoretical value given in the example in Section \ref{sec_example1} (represented as a dashed yellow line).

As a concrete example for $n=3$ we consider the perturbation $g(\tau)=20\cos(5\tau)+16\cos(3\tau)$. In this case $g^{[\pm 1]}=0$, $g^{[\pm 5]}=10$ and $g^{[\pm 10]}=8$. By definition~\eqref{def_Gn}  of $G_n$ we have that

$$
\begin{dcases}
G_1=\{-5,-3,3,5\}
\\
G_2=\{-10,-8,-6,-2,0,2,6,8,10\}
\\
G_3=\{-15,-13,-11,-9,-7,-5,-3,-1,1,3,5,7,9,11,13,15\}
.
\end{dcases}
$$
Since $1\notin G_1$, $1\notin G_2$ and $1\in G_3$, $n(g)=3$. We see on Table \ref{tbl_numerics} and on the right of Figure \ref{fig_numerics} the numerical approximation of $\chi_3^{[-1]}$ for several values of $\rho$. As in the case $n=2$, the value is stable between $\rho=4$ and $\rho=10$. Notice that for higher values of $\rho$ the values of $\chi_2^{[-1]}$ and $\chi_3^{[-1]}$ start deviating quickly from the stable value. This occurs as a result of large cancellations involved in the computations. As the range $\rho\in(4,10)$ yields an accurate result for $\chi_2^{[-1]}$, we take the stable numerical value in this interval as a valid approximation and we can conclude that


$$
\partial_u\Delta(u,\tau,\mu)=\frac{2e^{-\frac{\pi}{2\varepsilon}}}{\varepsilon^2}\cdot\left[
\Im\left( \chi^{[-1]}_3 \cdot e^{i(\tau-u/\varepsilon)}\right)\cdot\mu^3 +
\mathcal{O}(\mu^{4} )+
\mathcal{O}\left(|\mu|\cdot e^{-\frac{\pi}{2\varepsilon}}\right)+\mathcal{O}\left(\frac{|\mu|^3}{\log(1/\varepsilon)}\right)\right]
$$
is a valid asymptotic formula for the splitting with $\chi_3^{[1]}\approx 7$.

\begin{table}[H]
{
\small
\centering
\begin{tabular}{|*{11}{c|}}
\hline
$\rho$&4&5&6&7&8&9&10&11&12&14
\\
\hline
$\chi_2^{[-1]}$&55.7217&55.8103&55.8006&55.8053&55.7987&56.0248&56.7904&59.1055&64.3010&112.6084
\\
\hline
$\chi_3^{[-1]}$&6.7623&7.0372&7.0692&7.0901&7.1007&7.1193&7.2627&7.2962&8.0834&9.5561
\\
\hline
\end{tabular}
}
\caption{Numerical approximation for $\chi_2^{[-1]}$ and $\chi_3^{[-1]}$ for different values of $\rho$ with $\Re(z_0)=40$ and $N=20$.}
\label{tbl_numerics}
\end{table}

    \begin{figure}[H]
    \begin{minipage}{0.45\textwidth}
\includegraphics[width=\textwidth]{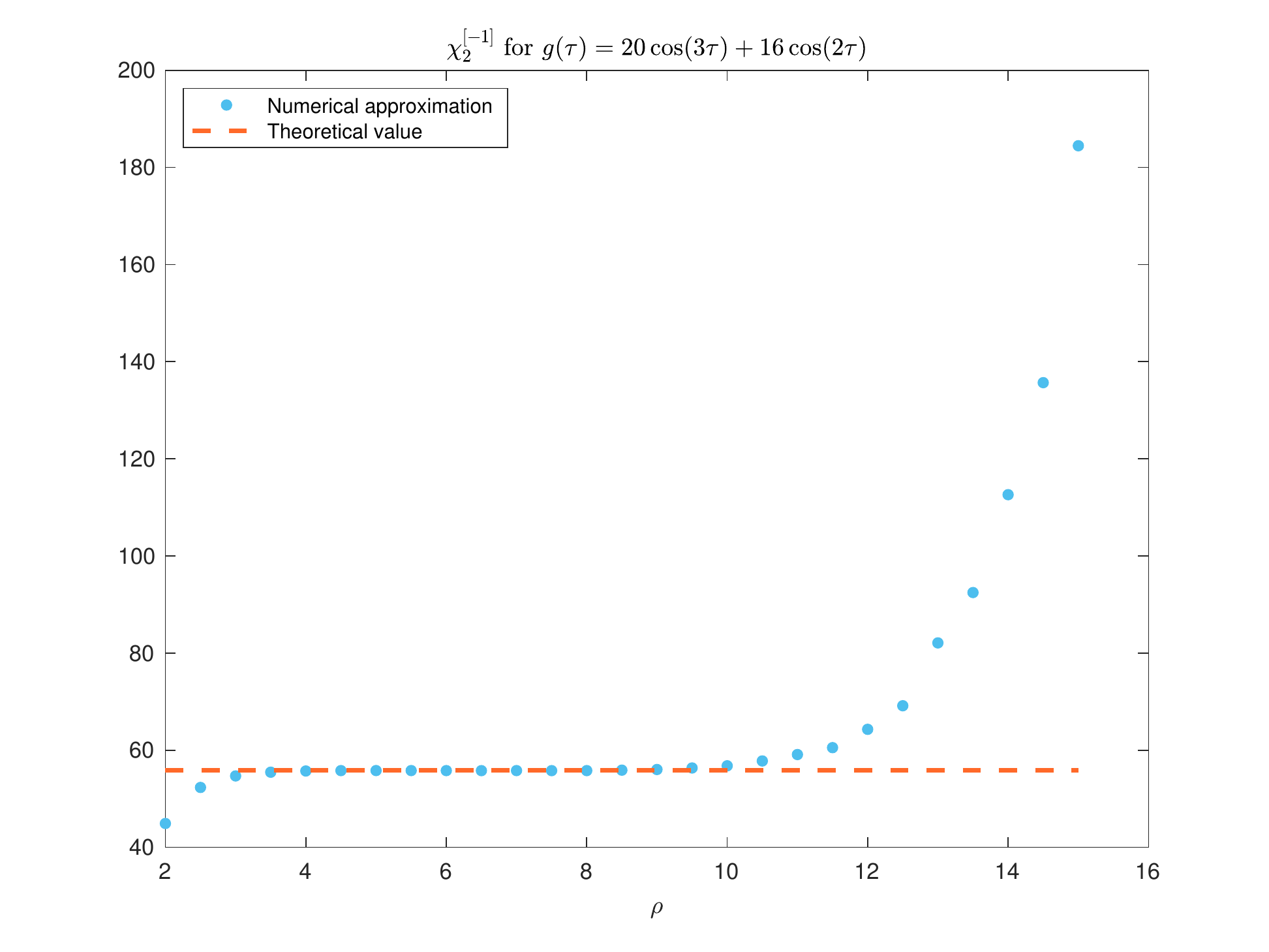}
    \end{minipage}
    \begin{minipage}{0.45\textwidth}
\includegraphics[width=\textwidth]{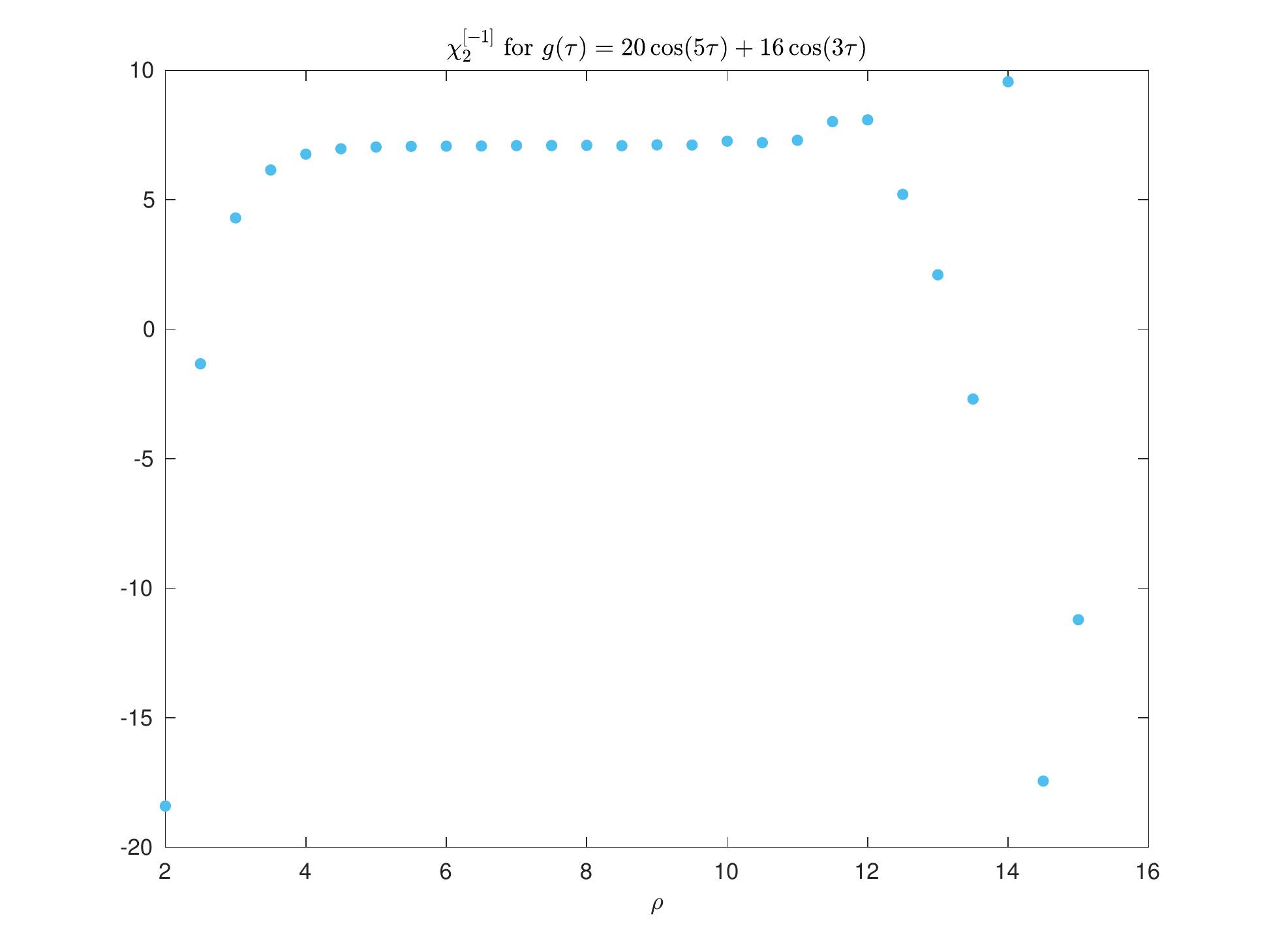}
    \end{minipage}
    \caption{Numerical approximation for $\chi_2^{[-1]}$ and $\chi_3^{[-1]}$ for different values of $\rho$ with $\Re(z_0)=40$ and $N=20$.}
    \label{fig_numerics}
    \end{figure}

\section{Proof of \ref{teo_main_second}}
\label{sec_proof_theorem_main_second}

Take $g$ periodic and let $n=n(g)$ defined in~\eqref{eq_definition_n}, namely $1\in G_n$ and $1\notin G_j$ for $j=1, \cdots, n-1$. Another way to express this condition is that $g\in \mathcal{E}_{n-1}$, see Remark \ref{rem:codimension}. We also fix the constants $\mu_0, \theta, \varrho \geq \varrho_0$ and $\sigma$ such that Theorem \ref{teo_inner_existence} holds true. 
We will omit them throughout this section. 

We begin by setting the usual convention that we denote by $M$ a constant independent of $\mu$ and $\varepsilon$ which could change its value during the section.

Now we introduce some notation. Consider $\Delta_{\mathrm{in}}(z,\tau,\mu)$, $\chi^{[k]}(\mu)$ and $\mathfrak{g}(z,\tau,\mu)$ (defined in Theorem \ref{teo_inner_existence}) as analytic functions of $\mu$ in a neighborhood of $\mu=0$. Their expansion around $\mu=0$ are
$$\Delta_{\mathrm{in}}(z,\tau,\mu)=\sum_{j\geq 1}\Delta_{\mathrm{in},j}(z,\tau)\cdot\mu^j,\qquad  
\mathfrak{g}(z,\tau,\mu)=\sum_{j\geq 1}\mathfrak{g}_j(z,\tau)\cdot \mu^j, \qquad  \chi^{[k]}(\mu)=\sum_{j\geq 1}\chi^{[k]}_j\cdot\mu^j.
$$

Since $\Delta_{\mathrm{in},j}(z,\tau)$ and $\mathfrak{g}_j(z,\tau)$ are $2\pi$-periodic functions in $\tau$ they admit a Fourier expansion that we write as 
$$
\Delta_{\mathrm{in},j}(z,\tau)=\sum_{k\in\mathbb{Z}}\Delta_{\mathrm{in},j}^{[k]}(z)\cdot e^{ik\tau}, \qquad  \mathfrak{g}_j(z,\tau)=\sum_{k\in\mathbb{Z}}\mathfrak{g}_j^{[k]}(z)\cdot e^{ik\tau}.
$$

\subsection{Conditions for \texorpdfstring{$\Delta_{\mathrm{in},j}^{[1]}=0$}{TEXT}}
\label{section_conditions_for_delta_0}
 
Our first goal is to prove that $\Delta_{\mathrm{in},j}^{[1]}=0$ if $j<n=n(g)$. 
To this end we expand in power series of $\mu$ the solutions of the inner equation provided by Theorem \ref{teo_inner_existence}:
$
\psi^{\pm}(z,\tau,\mu)=\sum_{j\geq 1}\psi_j^\pm(z,\tau)\cdot\mu^j.
$
We plug this expansion into the inner equation \eqref{eq_inner_equation} and we obtain the equations for each coefficient $\psi_j^{\pm}$ by equating terms of $\mathcal{O}(\mu^j)$. 
These are, for $j=1$,
\begin{equation}
\partial_\tau\psi^\pm_1(z,\tau)+\partial_z\psi_1^\pm(z,\tau)=-\frac{2}{z^2}g(\tau)
\label{eq_inner_mu_1}
\end{equation}
and, for $j>1$,
\begin{equation}
\partial_{\tau}\psi_j^\pm(z,\tau)+\partial_z\psi_j^\pm(z,\tau)=\frac{1}{8}z^2\sum_{l=1}^{j-1}\partial_z\psi_l^\pm(z,\tau)\cdot\partial_z\psi_{j-l}^\pm(z,\tau).
\label{eq_inner_mu_2}
\end{equation}
As was pointed out in Theorem \ref{teo_inner_existence}, the boundary conditions for $\psi^\pm$ (and for $\psi_j^\pm$) are:
\begin{equation}
\lim_{\Re(z)\to\pm\infty}\partial_z\psi_j^{\pm}(z,\tau)=0.
\label{eq_bound_cond_inner_mu}
\end{equation}

The next lemma links the harmonics of the perturbation, $g^{[k]}$, to the harmonics of the solutions, $\psi_j^{\pm,[k]}$.

\begin{lemma}
 If $k\notin G_j$ (see \eqref{def_Gn}), then $\ \psi_{j}^{\pm,[k]}(z)=0$ and therefore $\Delta_{\mathrm{in},j}^{[k]}(z)\equiv 0$.
 
 As a consequence, since $1\notin G_j$ for $j=1,\cdots, n-1$, we have that $\Delta_{\mathrm{in}}^{[1]}(z,\mu) = \mathcal{O}(\mu^n)$. 
\label{lemma_Fourier_coeff_inner}
\end{lemma}

\begin{proof}

We prove this result by induction. 
Consider $j=1$. 
Expanding equation \eqref{eq_inner_mu_1} in Fourier series we have that
\begin{equation}
\label{eq_inner_order_1_Fourier}
ik\psi_{1}^{\pm,[k]}(z)+\partial_z\psi_{1}^{\pm,[k]}(z)=-\frac{2}{z^2}g^{[k]}.
\end{equation}
If $k\notin G_1$ (see \eqref{def_Gn}), then $g^{[k]}=0$. 
In this case the only solution of the above-mentioned equation is $\psi_1^{\pm,[k]}=C^\pm\cdot e^{-ikz}$. 
However, in order to satisfy the boundary conditions \eqref{eq_bound_cond_inner_mu}, $C^\pm=0$.

Now we take $j>2$ and assume that for $\nu<j$ the result holds true. Expanding \eqref{eq_inner_mu_2} in Fourier we obtain:
\begin{equation}
\label{eq_inner_order_2_Fourier}
ik\psi_{j}^{\pm,[k]}(z)+\partial_z\psi_{j}^{\pm,[k]}(z)=\frac{1}{8}z^2\sum_{l=1}^{j-1}\sum_{m\in\mathbb{Z}}\partial_z\psi_{l}^{\pm,[m]}(z)\cdot\partial_z\psi_{j-l}^{\pm,[k-m]}(z).
\vspace{0.1cm}
\end{equation}
By induction hypothesis, the non-zero terms on the right-hand side are those where $m\in G_l$ and $k-m\in G_{j-l}$. 
Therefore, $k=m+(k-m)=m_1+\dots+m_l+m'_1+\dots+m'_{j-l}$ with $m_i,m'_i\in G_1$. This means $k\in G_j$. 
Hence, if $k\notin G_j$ the right-hand side of the equation has to be 0 and, given the boundary condition, so is $\partial_z\psi_{j}^{\pm,[k]}$.

\vspace{0.3cm}

\end{proof}

\subsection{A formula for \texorpdfstring{$\chi^{[-1]}(\mu)$}{TEXT}}
We can now state the following result, which relates the Taylor coefficients $\chi^{[-1]}_j$ of $\chi^{[-1]}(\mu)$ with the corresponding Taylor coefficients of $\Delta_{\mathrm{in},j}^{[1]}(z,\mu)$.

\begin{lemma}
\label{lemma_simplified_formula_chi_j}
The coefficient $\chi^{[-1]}$ is given by:
$$
\chi^{[-1]}(\mu) = \lim_{z\to -i\infty} e^{iz} \cdot \Delta_{\mathrm{in}}^{[1]}(z,\mu).
$$
As a consequence, by the analyticity with respect to $\mu$ we obtain
\begin{equation}
\chi_j^{[-1]}=\lim_{z\to-i\infty}e^{iz}\cdot\Delta_{\mathrm{in},j}^{[1]}(z).
\label{eq_lim_chi_1}
\end{equation}
\end{lemma}
\begin{proof}
Let us recall formula \eqref{eq_def_Delta_inner} for $\Delta_{\mathrm{in}}(z,\tau,\mu)$ in Theorem  \ref{teo_inner_existence}:
$$
\Delta_{\mathrm{in}}(z,\tau,\mu)=\sum_{k<0}\chi^{[k]}(\mu)\cdot e^{ik(z-\tau+\mu\mathfrak{g}(z,\tau,\mu))}.
$$
We can write
$$\Delta_{\mathrm{in}}(z,\tau,\mu)=\sum_{k<0}\chi^{[k]}(\mu)\cdot e^{ik(z-\tau)}+\sum_{k<0}\chi^{[k]}(\mu)\cdot e^{ik(z-\tau)}\cdot\left(e^{ik\mu\mathfrak{g}(z,\tau,\mu)}-1\right)=:\Delta_{\mathrm{in},a}(z,\tau,\mu)+\Delta_{\mathrm{in},b}(z,\tau,\mu)
$$ 
and we have that \begin{equation}\Delta_{\mathrm{in}}^{[1]}(z,\mu)=\Delta_{\mathrm{in},a}^{[1]}(z,\mu)+\Delta_{\mathrm{in},b}^{[1]}(z,\mu)=\chi^{-1}(\mu)\cdot e^{-iz}+\Delta_{\mathrm{in},b}^{[1]}(z,\mu).
\label{eq_aux_chi_2}
\end{equation}

For each $k<0$ the elements in the term $\Delta_{{\mathrm{in}},b}$ satisfy: 

\begin{align*}
\label{eq_aux_chi_1}
&\left|\chi^{[k]}(\mu)\cdot e^{ik(z-\tau)}\cdot\left(e^{ik\mu\mathfrak{g}(z,\tau,\mu)}-1\right)\right|\leq \cdot\left|\chi^{[k]}(\mu)\right|\cdot \left|e^{ik(z-\tau)}\right|\cdot |k \mu\mathfrak{g}(z,\tau,\mu)|\cdot e^{|k\mu\mathfrak{g}(z,\tau,\mu)|}\leq\\
&\left|\chi^{[k]}(\mu)\right|\cdot e^{|k|(\Im(z)+|\mu\mathfrak{g}(z,\tau,\mu)|)}\cdot|k \mu\mathfrak{g}(z,\tau,\mu)|.
\end{align*}
By Theorem \ref{teo_inner_existence} we know that $|\mathfrak{g}(z,\tau,\mu)|\leq M|z|^{-1}$. In particular, for $z$ with $-\Im(z)$ big enough we have $|\mathfrak{g}(z,\tau,\mu)|<1$ and $\Im(z)+|\mathfrak{g}(z,\tau,\mu|<0$. Therefore,

\begin{align*}
|\Delta_{\mathrm{in},b}|&\leq \sum_{k<0}\left|\chi^{[k]}(\mu)\right|\cdot e^{|k|(\Im(z)+|\mu\mathfrak{g}(z,\tau,\mu)|)}\cdot|k \mu\mathfrak{g}(z,\tau,\mu)|\leq M
\cdot e^{-\Im(z)}\cdot\left|\mu\mathfrak{g}(z,\tau,\mu)\right|.
\end{align*}
Hence, $\lim_{z\to-i\infty}\left|e^{iz}\cdot \Delta_{\mathrm{in,b}}\right|=0$ and, in particular, $\lim_{z\to-i\infty}e^{iz}\cdot\Delta_{\mathrm{in},b}^{[1]}=0$. 
Taking the limit in \eqref{eq_aux_chi_2} the result follows.

\end{proof}

Item 1 and 2 of Theorem \ref{teo_main_second} are a straightforward consequence of \ref{lemma_simplified_formula_chi_j} 
and \ref{lemma_Fourier_coeff_inner}.

\subsection{Computing \texorpdfstring{$\chi_1^{[-1]}$}{TEXT} and \texorpdfstring{$\chi_2^{[-1]}$}{TEXT}} 

To finish the proof of Theorem \ref{teo_main_second}, we present the explicit calculations for $\chi_1^{[-1]}$ and $\chi_2^{[-1]}$. 
For this, we integrate explicitly \eqref{eq_inner_mu_1} and \eqref{eq_inner_mu_2} to obtain $\psi_1^{\pm,[1]}$ and $\psi_2^{\pm,[1]}$, respectively. 
We subtract the stable and unstable solutions to get the first two terms of the Taylor expansion in $\mu$ of $\Delta_{\mathrm{in}}^{[1]}(\mu)$ and we apply Lemma \ref{lemma_simplified_formula_chi_j}.

\subsubsection{Computation of \texorpdfstring{$\mathbf{\chi_1^{[-1]}}$}{TEXT}}

We solve equation \eqref{eq_inner_mu_1} to obtain:
\begin{equation}
\psi_1^\pm(z,\tau)=-2\sum_{k\in\mathbb{Z}}g^{[k]}\int_{\pm\infty}^0\frac{1}{(z+t)^2}e^{ik(\tau+t)}dt.
\label{eq_sol_inner_j_1}
\end{equation}
Therefore
$$
\Delta_{\mathrm{in},1}(z,\tau)=-2\sum_{k\in\mathbb{Z}} e^{ik(\tau-z)}g^{[k]}\int_{-\infty+z}^{+\infty+z}\frac{e^{ikt}}{t^2}.
$$
We compute the integral by residues. As $\Im(z)<0$:
$$
\int_{-\infty+z}^{+\infty+z}\frac{e^{imt}}{t^2}dt=\left\{\begin{array}{ccl}-2m\pi && \text{if }m>0 \\ 0 & & \text{if }m\leq 0\end{array}\right.
$$
so we obtain:
\begin{equation}
\Delta_{\mathrm{in},1}(z,\tau)=\sum_{k>0}e^{ik(\tau-z)}g^{[k]}4k\pi.
\label{eq_second_comparison_order_1}
\end{equation}
From here the first Fourier coefficient is
$
\Delta_{\mathrm{in},1}^{[1]}=4\pi e^{-iz}g^{[1]}.
$
Now we can apply Lemma \ref{lemma_simplified_formula_chi_j}:
\begin{equation}
\chi_1^{[-1]}=\lim_{z\to-i\infty}e^{iz}\cdot\Delta_{\mathrm{in},1}^{[1]}=4\pi g^{[1]}.
\label{formula_order_mu}
\end{equation}


\subsubsection{Computation of \texorpdfstring{$\mathbf{\chi_2^{[-1]}}$}{TEXT}}

Following the same scheme we integrate \eqref{eq_inner_mu_2} for $j=2$. We write the derivative of the solution \eqref{eq_sol_inner_j_1} as follows:
$$
\partial_z\psi_1^\pm(z,\tau)=4\cdot\sum_{k\in\mathbb{Z}}g^{[k]}\cdot e^{ik(\tau-z)}\int_{\pm\infty+z}^z\frac{e^{ikt}}{t^3}dt.
$$
Plugging this expression in \eqref{eq_inner_mu_2}, the equation becomes:

$$
\partial_\tau\psi_2^\pm(z,\tau)+\partial_z\psi_2^\pm(z,\tau)=2z^2\sum_{k\in\mathbb{Z}}e^{ik(\tau-z)}\sum_{l\in\mathbb{Z}}g^{[l]}\cdot g^{[k-l]}\left(\int_{\pm\infty}^z\frac{e^{ilt}}{t^3}dt\right)\left(\int_{\pm\infty}^z\frac{e^{i(k-l)t}}{t^3}dt\right).
$$
We solve the equation by integrating and reorganize the terms:

\begin{align*}
\psi_2^+(z,\tau)&=\int_{\pm\infty}^{0}2(z+s)^2\sum_{k\in\mathbb{Z}}e^{ik(\tau-z)}\sum_{l\in\mathbb{Z}}g^{[l]}\cdot g^{[k-l]}\left(\int_{\pm\infty}^{z+s}\frac{e^{ilt}}{t^3}dt\right)\left(\int_{\pm\infty}^{z+s}\frac{e^{i(k-l)t}}{t^3}dt\right)ds=\\
&=2\sum_{k\in\mathbb{Z}}e^{ik(\tau-z)}\sum_{l\in\mathbb{Z}}g^{[l]}\cdot g^{[k-l]}\int_{\pm\infty}^zs^2\left(\int_{\pm\infty}^{s}\frac{e^{ilt}}{t^3}dt\right)\left(\int_{\pm\infty}^{s}\frac{e^{i(k-l)t}}{t^3}dt\right)ds.
\end{align*}

From this formula we extract the first harmonic by taking the term $k=1$. Note that we can multiply by $2$ and take the sum  over $l$ only for $l>1$.

$$
 e^{iz}\cdot\psi_2^{\pm,[1]}=4\sum_{l>1}g^{[l]}\cdot g^{[1-l]}\int_{\pm\infty}^zs^2\left(\int_{\pm\infty}^{s}\frac{e^{ilt}}{t^3}dt\right)\cdot \left(\int_{\pm\infty}^{s}\frac{e^{i(1-l)t}}{t^3}dt\right)ds.
$$
Integration by parts yields:
\begin{align*}
    \int_{\pm\infty}^z
s^2&\left(\int_{\pm\infty}^s\frac{e^{i(1-l)t}}{t^3}dt\right)\left(\int_{\pm\infty}^{s}\frac{e^{ilt}}{t^3}dt\right)ds=\left.\frac{1}{3}s^3\left(\int_{\pm\infty}^{s}\frac{e^{i(1-l)t}}{t^3}dt\right)\left(\int_{\pm\infty}^{s}\frac{e^{ilt}}{t^3}dt\right)\right]_{\pm\infty}^z\\&-\frac{1}{3}\int_{\pm\infty}^ze^{i(1-l)s}\left(\int_{\pm\infty}^{s}\frac{e^{ilt}}{t^3}dt\right)ds-\frac{1}{3}\int_{\pm\infty}^ze^{ils}\left(\int_{\pm\infty}^s\frac{e^{i(1-l)t}}{t^3}dt\right)ds=
\\
=&\frac{1}{3}z^3\left(\int_{\pm\infty}^{z}\frac{e^{(1-l)t}}{t^3}dt\right)\left(\int_{\pm\infty}^{z}\frac{e^{ilt}}{t^3}dt\right)-\frac{1}{3}\int_{\pm\infty}^ze^{i(1-l)s}\left(\int_{\pm\infty}^{s}\frac{e^{ilt}}{t^3}dt\right)ds
\\
&-\frac{1}{3}\int_{\pm\infty}^ze^{ils}\left(\int_{\pm\infty}^s\frac{e^{i(1-l)t}}{t^3}dt\right)ds=:\frac{1}{3}(A-B-C).
\end{align*}
Since $l>1$, further integration by parts leads to:
\begin{align*}
B=\int_{\pm\infty}^ze^{i(1-l)s}\left(\int_{\pm\infty}^{s}\frac{e^{ilt}}{t^3}dt\right)ds&=\frac{1}{i(1-l)}\left.e^{i(1-l)s}\int_{\pm\infty}^s\frac{e^{ilt}}{t^3}\right]_{\pm\infty}^{z}-\frac{1}{i(1-l)}\int_{\pm\infty}^z\frac{e^{is}}{s^3}ds=\\
&=\frac{1}{i(1-l)}e^{i(1-l)z}\int_{\pm\infty}^z\frac{e^{ilt}}{t^3}dt-\frac{1}{i(1-l)}\int_{\pm\infty}^z\frac{e^{is}}{s^3}ds
\end{align*}
and
\begin{align*}
C=\int_{\pm\infty}^ze^{ils}\left(\int_{\pm\infty}^{s}\frac{e^{i(1-l)t}}{t^3}dt\right)ds&=\frac{1}{il}\left.e^{ils}\int_{\pm\infty}^s\frac{e^{i(1-l)t}}{t^3}\right]_{\pm\infty}^{z}-\frac{1}{il}\int_{\pm\infty}^z\frac{e^{is}}{s^3}ds=\\
&=\frac{1}{il}e^{ilz}\int_{\pm\infty}^z\frac{e^{i(1-l)t}}{t^3}dt-\frac{1}{il}\int_{\pm\infty}^z\frac{e^{is}}{s^3}ds.
\end{align*}
Gathering the formulas:

\begin{align*}
e^{iz}\cdot\psi^{\pm,[1]}_{2}(z)=\frac{4}{3}\sum_{l>1}g^{[1-l]}g^{[l]}&\left\{z^3\int_{\pm\infty}^{z}\frac{e^{i(1-l)t}}{t^3}dt\int_{\pm\infty}^{z}\frac{e^{ilt}}{t^3}dt
-\frac{1}{i(1-l)}e^{i(1-l)z}\int_{\pm\infty}^z\frac{e^{ilt}}{t^3}dt\right.\\
+&\left.\frac{1}{i(1-l)}\int_{\pm\infty}^z\frac{e^{is}}{s^3}
-\frac{1}{il}e^{ilz}\int_{\pm\infty}^z\frac{e^{i(1-l)t}}{t^3}dt+\frac{1}{il}\int_{\pm\infty}^z\frac{e^{is}}{s^3}ds\right\}.
\end{align*}
Now we subtract $\psi_2^{-,[1]}$ and $\psi_2^{+,[1]}$. We first point out that 
$$
\int_{\pm\infty}^z\frac{e^{i(1-l)t}}{t^3}dt=\int_{-i\infty}^z\frac{e^{i(1-l)t}}{t^3}.
$$
Indeed, the integrand converges exponentially as $z\to-i\infty$ (since  $(1-l)<0$) and there are no singularities in $\{\Im(z)<0\}$ (note that the paths that join $-\infty$ and $\infty$ in the previous integral are in the region $\{\Im(z)<0\}$). 
Hence, we can change paths of integration. 
On the other hand, integrating by residues:
$$
\int_{-\infty}^\infty\frac{e^{imt}}{t^3}dt=
\left\{\begin{array}{ccl}
-i\pi m^2&&m>0\\
0&&m\leq 0
\end{array}\right..
$$
With these claims:

\begin{align*}
e^{iz}&\cdot\Delta_{\mathrm{in},2}^{[1]}(z,\tau)
=e^{iz}\cdot(\psi_{2}^{-,[1]}-\psi_2^{+,[1]})
=+\frac{4}{3}\sum_{l>1}g^{[l]}g^{[1-l]}\left\{z^3\int_{-i\infty}^z\frac{e^{i(1-l)t}}{t^3}dt\cdot\left(
\int_{-\infty}^{+\infty}\frac{e^{ilt}}{t^3}dt
\right)\right.\\
&\left.-\frac{1}{i(1-l)}e^{i(1-l)}\int_{-\infty}^{+\infty}\frac{e^{ilt}}{t^3}dt+\frac{1}{i(1-l)}\int_{-\infty}^{+\infty}\frac{e^{is}}{s^3}-\frac{1}{il}e^{ilz}\int_{-\infty}^{+\infty}\frac{e^{i(1-l)t}}{t^3}+\frac{1}{il}\int_{-\infty}^{+\infty}\frac{e^{is}}{s^3}ds\right\}\\
&
=+\frac{4}{3}\sum_{l>1}g^{[l]}\cdot g^{[1-l]}\left\{-i\pi l^2z^3\int_{-i\infty}^z\frac{e^{i(1-l)t}}{t^3}dt+\frac{\pi l^2}{(1-l)}e^{i(1-l)z}-\frac{\pi}{(1-l)}-\frac{\pi}{l}\right\}.
\end{align*}
We apply now Lemma \ref{lemma_simplified_formula_chi_j}. For that we take $z\to-\infty$ and use the following inequality (valid for $k>1$):
$$
\left|z^3\int_{-i\infty}^z\frac{e^{i(1-k)t}}{t^3}dt\right|\leq M\cdot\frac{|e^{i(1-k)z}||z|^3}{|z|^3}=M\cdot|e^{i(1-k)z}|,
$$
whence
$$
\lim_{z\to-i\infty}z^3\int_{-i\infty}^z\frac{e^{i(1-k)t}}{t^3}dt=0.
$$
Therefore, we obtain:
\begin{align}
\chi_2^{[-1]}&=\lim_{z\to-i\infty}e^{iz}\cdot\Delta_{\mathrm{in},2}^{[1]}
=-\frac{4}{3}
\sum_{k>0}g^{[k]}\cdot g^{[1-k]}\left(\frac{1}{(1-k)}+\frac{1}{k}\right)\pi
=-\frac{4\pi}{3}\sum_{k>1}\frac{g^{[k]}\cdot g^{[1-k]}}{k(1-k)}.
\label{formula_order_mu_2}
\end{align}

\section{Splitting formula: Proof of \ref{th:maintheorem}}\label{sec_app_inner_equation}

As in the previous section we fix $g$ a periodic function with $n=n(g)$, see~\eqref{eq_definition_n}. We still use $M$ to refer to constants independent of $\mu$ and $\varepsilon$.

The following straightforward consequence of Cauchy's theorem will be used along this section without explicit mention.  
\begin{lemma}
Let $\eta>0$, $\nu \in \mathbb{N}$ and $h:B_\eta\subset\mathbb{C}\longrightarrow\mathbb{C}$ be an analytic function such that $|h(\mu)|\leq M_h$. Let us write $h(\mu)=\sum_{j\geq 0}h_j\cdot\mu^j$ for the power expansion around $\mu=0$. There exists one constant $M_1$ depending on $\nu$ such that for all $j=0,\cdots,\nu$, we have that $
|h_j|\leq M_1\cdot M_h$.
\label{lemma_bounds_coefficients_powerseries}
\end{lemma}

\subsection{Preliminaries and heuristics of the proof} 
 To proceed with the proof of Theorem \ref{th:maintheorem} we need first to state some previous results about the splitting of separatrices of system~\eqref{eq_model} which can be found in~\cite{bib_exp_small}. To this end, let us first introduce some notation and setting. 

Recall that  the functions $\widehat{T}^{\mathrm{u},\mathrm{s}}(u,\tau,\mu)$ are defined in \eqref{eq:tus} in terms of the generating functions $S^{s,u}(x,\tau)$ which satisfy the Hamilton-Jacobi equation \eqref{eq:HJnova}. 
Moreover, using this equation, the relation
\eqref{eq_chain_rule} and that  $\widehat{T}^{\mathrm{u},\mathrm{s}}(u,\tau,\mu)=T_0(u)+T^{\mathrm{u},\mathrm{s}}(u,\tau,\mu)$ (see \eqref{eq:T}), one easily obtains the equation for $T^{\mathrm{u},\mathrm{s}}$

\begin{equation}
\frac{1}{\varepsilon}\partial_{\tau}T^{\mathrm{u},\mathrm{s}}(u,\tau,\mu)+\partial_uT^{\mathrm{u},\mathrm{s}}(u,\tau,\mu)
=
-\frac{1}{8}\cosh^2(u)(\partial_uT^{\mathrm{u},\mathrm{s}})^2(u,\tau,\mu)+2\mu \frac{g(\tau)}{\cosh^2(u
)}.
\label{eq_HJ}
\end{equation}
The solutions that describe the stable and unstable manifolds are characterized by being $2\pi$-periodic in $\tau$ and satisfying the boundary conditions (see \eqref{eq_boundarcond_T})
\begin{equation}
\lim_{\Re(u)\to\mp\infty}\cosh^2(u)\cdot\partial_uT^{\mathrm{u},\mathrm{s}}(u,\tau,\mu)=0.
\label{eq_bound_cond}
\end{equation}

These solutions are well understood (\cite{bib_exp_small}): they are known to exist in suitable complex domains, to be analytic in all variables, $2\pi$-periodic in $\tau$ and to present exponential decay as $\Re(u)\to \pm \infty$. Furthermore, they can be analytically continued to complex regions reaching $\varepsilon$-neighborhoods of the singularities of the unperturbed homoclinic trajectory (see~\eqref{eq_homoclinica}) closest to the real axis, that is, $u=\pm i\frac{\pi}{2}$.

Since we want to study the difference between these solutions, we only need to know how they behave in a common domain. Fix $\phi\in(0,\pi/2)$ and take the following domain (see \ref{fig_outer_domains}): 
\begin{equation}
\mathcal{D}_{\domvar,\phi}^{\mathrm{out}}=\left\{u\in\mathbb{C};|\Im(u)|<-\tan(\phi)\cdot \Re(u)+\frac{\pi}{2}-\domvar\varepsilon,|\Im(u)|<\tan(\phi)\cdot \Re(u)+\frac{\pi}{2}-\domvar\varepsilon, \right\},
\label{eq_definition_outer_domains}
\end{equation}
with $\domvar>0$. 
Moreover, as we want to keep track of the analyticity
respect to $\mu$, from now on we will take $\mu \in B_{\mu_0}$, the complex ball centered at $0$ and of radius $ \mu_0$.

\begin{figure}[H]
    \centering
    \begin{overpic}[width=0.6\textwidth]{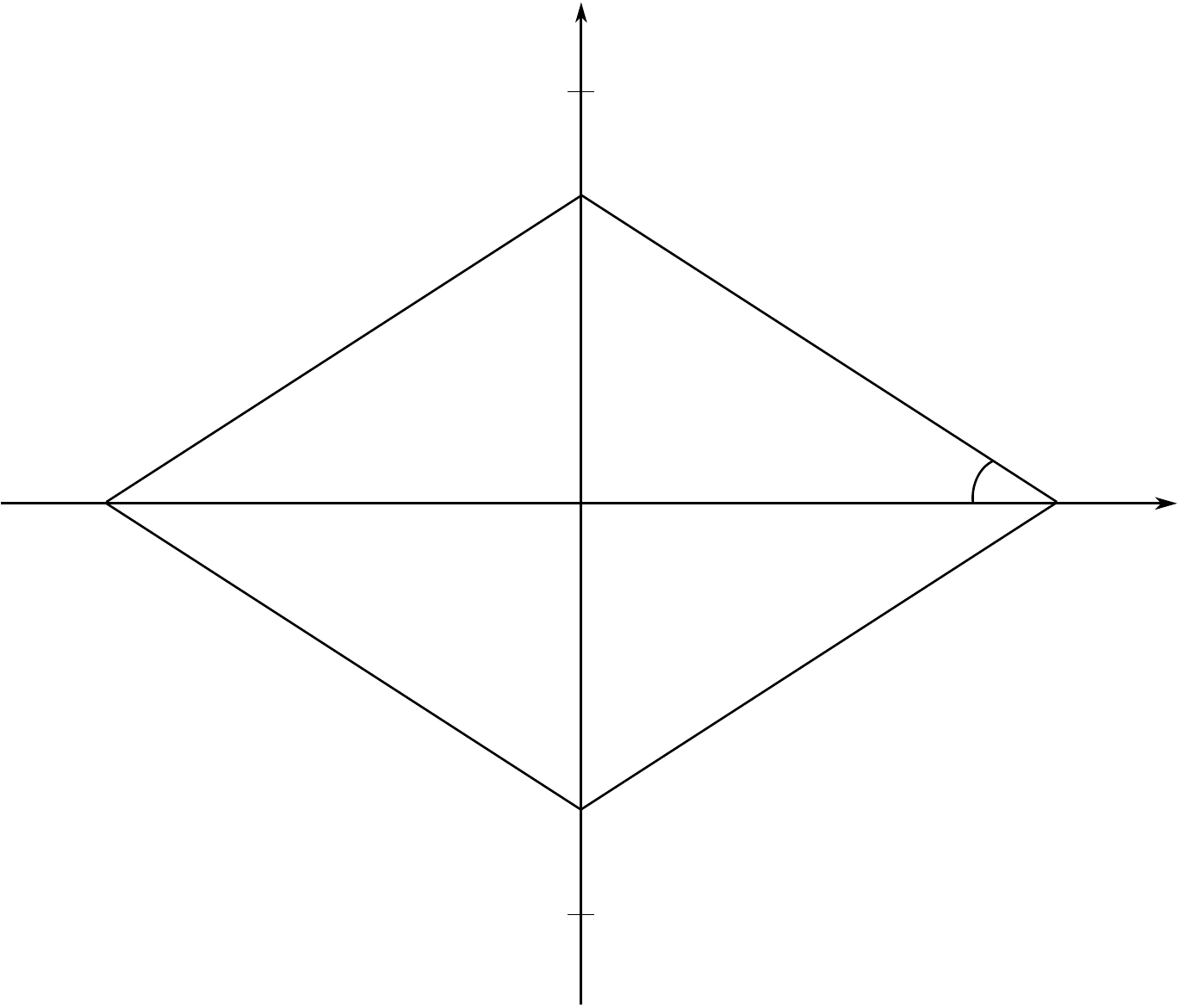}
    \put(52,77){$i\frac{\pi}{2}$}
    \put(52,7){$-i\frac{\pi}{2}$}
    \put(52,68){$i\left(\frac{\pi}{2}-\domvar\varepsilon\right)$}
    \put(79,45){$\phi$}
    \put(25,47){$\mathcal{D}^{\mathrm{out}}_{\domvar,\phi}$}
    \end{overpic}
    \caption{Domain $\mathcal{D}^{\mathrm{out}}_{\domvar,\phi}$.}
    \label{fig_outer_domains}
\end{figure}

In this domain we can formulate the following theorem by paraphrasing Theorem $4.4$ in \cite{bib_exp_small}:

\begin{theorem}(\cite{bib_exp_small}) 
Fix $\phi\in \left (0,\frac{\pi}{2}\right )$ and $\mu_0>0$. There exist $\domvar_0>0$ and $\varepsilon_0>0$ 
such that $\forall \mu \in B_{\mu_0}$, $\forall\varepsilon\in(0,\varepsilon_0)$ and  $\forall \domvar \ge \domvar_0$ satisfying that $\varepsilon \domvar <1$, the Hamilton-Jacobi equation \eqref{eq_HJ} has solutions $T^{\mathrm{u},\mathrm{s}}(u,\tau,\mu)$ analytic in $u,\tau,\mu$ and periodic in $\tau$ satisfying the boundary conditions \eqref{eq_bound_cond} such that they are defined in the domain $\mathcal{D}_{\domvar,\phi}^{\mathrm{out}}\times \T_\sigma\times B_{\mu_0}$ and in this domain the following bound holds:
$$
|\partial_uT^{\mathrm{u},\mathrm{s}}(u,\tau,\mu)|\leq M\cdot\frac{|\mu|\cdot\varepsilon}{|u^2+(\pi/2)^2|^3}.
$$
\label{teo_existence_outer}
\end{theorem}

Next theorem, which is an an adaptation of the results from~\cite{bib_exp_small} (Proposition 4.22, Theorem 4.23, Corollary 4.24), 
gives a characterization of the difference $\Delta(u,\tau,\mu)$. 

\begin{theorem} 
Under the same assumptions of Theorem \ref{teo_existence_outer},  
there exists a real analytic function $\mathcal{C}(u,\tau,\mu)$ defined in $\mathcal{D}_{\domvar,\phi}^{\mathrm{out}}\times \T_\sigma \times B_{\mu_0}$ satisfying the following bounds:
\begin{equation}\label{boundsCcal}
|\mathcal{C}(u,\tau,\mu)|\leq M\cdot\frac{|\mu|\cdot\varepsilon}{|u^2+(\pi/2)^2|}, \qquad 
    |\partial_u\mathcal{C}(u,\tau,\mu)|\leq M\cdot\frac{|\mu|}{\domvar\cdot{|u^2+(\pi/2)^2|}},
\end{equation}
and such that the difference between the parameterizations $T^{\mathrm{u},\mathrm{s}}$ of the stable and unstable manifolds is given in $\mathcal{D}_{\domvar,\phi}^{\mathrm{out}}\times \T_\sigma \times B_{\mu_0}$ by the expression:
\begin{equation}
    \Delta(u,\tau,\mu)=\sum_k\Upsilon^{[k]}(\mu)\cdot e^{ik\left(u/\varepsilon-\tau+\mathcal{C}(u,\tau,\mu)\right)},\label{eq_formula_diff_outer}
\end{equation}
where $\Upsilon^{[k]}(\mu)$ are analytic functions of $\mu \in B_{\mu_0}$.

In addition, for $u\in \mathbb{R} \cap \mathcal{D}_{\domvar,\phi}^{\mathrm{out}}$, $\tau\in \mathbb{R}$, $\mu \in B_{\mu_0}$ and $\varepsilon \in (0,\varepsilon_0)$ we have that

\begin{equation}
\partial_u\Delta(u,\tau,\mu)=
\frac{2e^{-\frac{\pi}{2\varepsilon}}}{\varepsilon^2}\cdot \left[\Im\left(\chi^{[-1]}(\mu)\cdot e^{i\left(\tau-u/\varepsilon\right)}\right)+\mathcal{O}\left(\frac{|\mu|}{\log(1/\varepsilon)}\right)\right], 
\label{eq_known_splitting_inner}
\end{equation}
where $\chi^{[-1]}(\mu)$ is defined in Theorem \ref{teo_inner_existence}.
\label{teo_difference_manifolds}
\end{theorem}
Note that, even though the existence and properties of the function $\mathcal{C}$ are proved in \cite{bib_exp_small}, in our case we can derive the sharper bound \eqref{boundsCcal}, whose prove we leave to \ref{sec_appendix_proof_c}.

From now on we fix $\mu_0,\varepsilon_0,\phi,\varrho\geq \varrho_0$ and $\sigma$ such that Theorems \ref{teo_existence_outer} and \ref{teo_difference_manifolds} hold true. As usual we will omit the dependence on these constants. We emphasize that, since 
$\mathcal{D}^{\mathrm{out}}_{\domvar_1,\phi} \subset \mathcal{D}^{\mathrm{out}}_{\domvar_2,\phi}$ when $\domvar_1\geq \domvar_2$, we can (and we will) take $\domvar_0$ as big as we need in our proofs. 

To finish this section, we define the analytic expansion of $\Delta(u,\tau,\mu)$  and $\Upsilon^{[k]}(\mu)$ around $\mu=0$:
\begin{equation}
\Delta(u,\tau,\mu)=\sum_{j\geq 0}\Delta_j(u,\tau)\cdot \mu^j, \qquad 
 \Upsilon^{[k]}(\mu)= \sum_{j\geq 0} \Upsilon^{[k]}_j \cdot \mu^j,
\label{eq_Taylor_expansion_Delta}
\end{equation}
and the Fourier expansions of $\Delta(u,\tau,\mu)$ and $ \Delta_j(u,\tau)$ :
\begin{equation}\label{eq_Fourier_Deltaj}
\Delta(u,\tau,\mu) = \sum_{k\in \mathbb{Z}} \Delta^{[k]}(u,\mu)\cdot e^{ik\tau},\qquad 
\Delta_j(u,\tau)=\sum_{k\in \mathbb{Z}} \Delta_j^{[k]}(u) \cdot e^{ik \tau}.
\end{equation}

\subsubsection{Heuristics and strategy of the proof}
The aforementioned known splitting formula~\eqref{eq_known_splitting_inner} has the Fourier coefficient  $\chi^{[-1]}(\mu)$ in its leading term. 
In Section \ref{sec_proof_theorem_main_second} we have already analyzed the solutions of the inner equation in order to build insight into the expansion in powers of $\mu$ of $\chi^{[-1]}(\mu)$ and we have concluded in Theorem \ref{teo_main_second} that $\chi^{[-1]}(\mu)= \chi^{[-1]}_n  \mu^n + \mathcal{O}(\mu^{n+1})$. 
This analysis suggests that the leading term of the splitting is of order $\mu^n$. 
However,  taking the first non-vanishing term of $\chi^{[1]}$ as a first term in the asymptotic expression does not make \eqref{eq_known_splitting_inner} a valid asymptotic expression straightforwardly if $n\geq 2$. 
Indeed, in this case, expression \eqref{eq_known_splitting_inner} becomes:

\begin{equation}
\partial_u\Delta(u,\tau,\mu)=
\frac{2e^{-\frac{\pi}{2\varepsilon}}}{\varepsilon^2}\cdot \left[\Im\left(\chi^{[-1]}_n\cdot e^{i\left(\tau-u/\varepsilon\right)}\right)\cdot\mu^n+\mathcal{O}(|\mu|^{(n+1)})+\mathcal{O}\left(\frac{|\mu|}{\log(1/\varepsilon)}\right)\right].
\label{eq_inner_replacing_directly}
\end{equation}
As $n \geq 2$, it is not clear that the main term dominates over the last error term when $\varepsilon, \mu$ are small, 
which invalidates it as a useful asymptotic expression (for instance when $|\mu|\ll |\log \varepsilon|^{-\frac{1}{n-1}}$, as happens in the classical case $\mu=\varepsilon$ or, more generally, $\mu=\mathcal{O}(\varepsilon^m)$, $m>0$).

Our strategy is to prove that the error term in \eqref{eq_inner_replacing_directly} is actually smaller. 
The fact that the leading term of the splitting formula is of order $\mu^n$ does not mean that the splitting function itself does not have terms of order $\mu^j$ for $j<n$; however, we will see that these terms turn out to be much smaller in $\varepsilon$ (in fact, exponentially smaller). 
This idea is simple and we can present it via this example: assume we had a quasiperiodic function, $f$, with an expansion:
\begin{equation}
f(\tau-u/\varepsilon,\mu)=\sum_{j\geq 1}\mu^j\sum_{k\in \mathbb{Z}} A_{j,k}\cdot e^{ik\left(\tau-u/\varepsilon\right)},
\label{eq_example}
\end{equation}
with $A_{j,k}\in\mathbb{C}$. 
Assume $|f|\leq M$ for $u\in\{z\in\mathbb{C},|\mathcal{I}(z)|<a\}$ and $\tau\in\mathbb{T}$. 
In this heuristic example we assume that $\varepsilon $ and $\mu$ are small parameters but $\mu$ is bigger than $e^{-\frac{a}{\varepsilon}}$ (the "natural" setting is $\mu=\mathcal{O}(\varepsilon^m)$, $m>0$; when  $\mu$ is exponentially small in $\varepsilon$, the splitting can be analyzed by classical perturbation theory). 

By using the fact the function $f$ is bounded in a complex strip, we can show that $A_{j,k}=\mathcal{O}(e^{-|k|\frac{a}{\varepsilon}})$.
Obviously, the first exponentially small term of order $e^{-\frac{a}{\varepsilon}}$ is given by the first $\ell $ for which $A_{\ell ,\pm1}\neq 0$, and hence $f\sim \mathcal{O}\left(\mu^\ell\cdot e^{-\frac{a}{\varepsilon}}\right)$ for real values of $u$. 
The terms in $\mu^j$ with $1\leq j<\ell $ are present, but they are of size $\mathcal{O}\left(\mu^j\cdot e^{-|k|\frac{a}{\varepsilon}}\right)$ with $|k|\geq 2$. 
Thus, they are much smaller, and the term $\mu^\ell $ dominates the splitting. 

In order to apply this idea we we split the power expansion of $\Delta$ in \eqref{eq_Taylor_expansion_Delta} as 
\begin{equation}\label{decompositionDelta_heuristics}
\Delta(u,\tau,\mu)= \Delta^{<n}(u,\tau,\mu) + \Delta^{\geq n}(u,\tau,\mu) :=\sum_{0<j<n}\Delta_j(u,\tau)\cdot\mu^j+\sum_{j\geq n}\Delta_j(u,\tau)\cdot\mu^j.
\end{equation}

To prove Theorem \ref{th:maintheorem} we follow the following strategy:
\begin{enumerate} 
\item 
We first prove, in Section \ref{sec_condition_delta_firstharm_0} $\Delta_j^{[\pm 1]}(u)\equiv 0$ if $1\leq j \leq n-1$ and therefore $\Delta^{[\pm 1]}(u,\mu)= \mathcal{O}(\mu^n)$. 
If the splitting distance were given by a formula like \eqref{eq_example}, it would be straightforward to conclude that the terms of lower order $\mathcal{O}(\mu^j)$ with $j<n$ are of higher exponentially small order. 
Although this is not the case, we have a similar formula, given by \ref{teo_difference_manifolds}:
\begin{equation}\label{heuristics_Delta}
\Delta(u,\tau,\mu)=\sum_{k\in \mathbb{Z}}\Upsilon^{[k]}(\mu )\cdot e^{ik\left(u/\varepsilon-\tau+\mathcal{C}(u,\tau,\mu)\right)},
\end{equation}
where $\C(u,\tau,\mu)$ is analytic in all arguments and bounded.

\item 
In Section \ref{sec_error_lower_order_terms} we analyze $\Delta^{<n}$ to establish that, loosely speaking, $\Delta^{<n}(u,\tau,\mu)=\mathcal{O}\left(\mu\cdot e^{-\frac{\pi}{2\varepsilon}\cdot 2}\right)$ (see Proposition \ref{prop_higher_exponential} for details). We work with identity~\eqref{heuristics_Delta}, the strategy being to perform a power series expansion in $\mu$ of $\Delta$, $\mathcal{C}$ and $\Upsilon^{[k]} = \sum_{j\geq 1} \Upsilon^{[k]}_j\cdot \mu^j$ and bound the constants $\Upsilon_j^{[k]}$, $j<n$. 
In Lemma \ref{lemma_bound_upsilon_no_tilde} we bound $\Upsilon_j^{[k]}$ for all $j\geq 1$ and any value of $k \in \mathbb{Z}$, and obtain, roughly speaking, that  
$ \Upsilon^{[k]}_j=\mathcal{O}(e^{-|k|\frac{\pi}{2\varepsilon}})$. In Lemma \ref{lemma_bound_upsilon_k_1} we improve the estimate in the case $k=1$: for $j<n$, we actually have $\Upsilon^{[\pm 1]}_j =\mathcal{O}(e^{-2\frac{\pi}{2\varepsilon}})$. Consequently we obtain, for $j<n$, that $|\Upsilon_j^{[k]}|$ is at least $\mathcal{O}(e^{-2\frac{\pi}{2\varepsilon}})$ for $k\in \mathbb{Z} \backslash \{0\}$.
In Proposition \ref{prop_higher_exponential} we prove that for real values of $u$ and $\tau$ the bounds of the coefficients transfer to $\Delta^{<n}$ and the desired bound is proven.

\item 
In Section \ref{sec_error_estimate_higher_order_terms} we 
analyze $\Delta^{\geq n}$. 
More precisely, the error term $\partial_u\Delta^{\geq n} - \partial_u \delta_0$ 
where $\delta_0$ is defined as:

\begin{equation}\label{eq_reminder_delta_0_heuristics}
\delta_0(u,\tau,\mu)=\frac{2e^{-\frac{\pi}{2\varepsilon}}}{\varepsilon}\cdot \Re \left (\chi^{[-1]}(\mu) \cdot e^{i \left (\tau - u/\varepsilon\right )} \right ).
\end{equation}
We recall that, by Theorem \ref{teo_main_second}, 
$\delta_0 = \mathcal{O}(\mu^n)$. 
Then, by Theorem \ref{teo_difference_manifolds}, we already know that this error term is  $\mathcal{O}\left (\frac{\mu} {\log(1/\varepsilon)\cdot \varepsilon^2}\cdot e^{-\frac{\pi}{2\varepsilon}}\right )$. 
Using an appropriate version of Schwartz lemma for analytic functions, we obtain the extra $\mu^n$ factor in the exponentially small bound of the error.

\end{enumerate}
 
\subsection{Condition for $\Delta_j^{[\pm 1]}=0$}
\label{sec_condition_delta_firstharm_0}
We will derive a condition that ensures certain harmonics of the Taylor coefficients of the solutions to the Hamilton-Jacobi equation  $\partial_uT^{\mathrm{u},\mathrm{s}}$ given by Theorem \ref{teo_existence_outer} are zero. 
We consider their Taylor expansions:
\begin{equation}
\partial_uT^{\mathrm{u},\mathrm{s}}(u,\tau,\mu)=\sum_{j\geq 1}\partial_uT_j^{\mathrm{u},\mathrm{s}}(u,\tau)\cdot\mu^j.
\label{eq_Taylor_expansion_T}
\end{equation}

\begin{lemma}  
 If $k\notin G_j$ (see \eqref{def_Gn}), then the $k$-th harmonic of $T_j^{\mathrm{u},\mathrm{s}}(u,\tau)$ satisfies $(\partial_u{T^{\mathrm{u},\mathrm{s}}_{j})}^{[k]}(u)\equiv 0$. 
 Hence, $\Delta_j^{[k]}(u)=0$. As a consequence, since $1\notin G_j$ for $j=1,\cdots,n-1$, we deduce that $\Delta^{[\pm 1]}(u,\mu)=\mathcal{O}(\mu^n)$. 
\label{lemma_Fourier_coeff_outer}
\end{lemma}
\begin{proof}
First notice that, by \eqref{eq_bound_cond}, we know that as $\Re(u)\to\pm\infty$ the functions  $\partial_uT^{\mathrm{u},\mathrm{s}}$ satisfy:
$$
|\partial_uT^{\mathrm{u},\mathrm{s}}(u,\tau,\mu)|\leq M\cdot e^{-2|\Re(u)|}.
$$
Using Lemma \ref{lemma_bounds_coefficients_powerseries}, we can state:
\begin{equation}
|\partial_uT_{j}^{\mathrm{u},\mathrm{s}}(u,\tau)|\leq M\cdot e^{-2|\Re(u)|},
\label{eq_init_cond_Fourier_mu}
\end{equation}
where $M$ depends on $j$.

Expanding the Hamilton-Jacobi equation \eqref{eq_HJ} in powers of $\mu$ we obtain for $j=1$

\begin{equation}
\frac{1}{\varepsilon}\partial_\tau T_{1}^{\mathrm{u},\mathrm{s}}(u,\tau)+\partial_uT_{1}^{\mathrm{u},\mathrm{s}}(u,\tau)=2\frac{g(\tau)}{\cosh^2(u)}
\label{eq_outer_mu_1}
\end{equation}
and, for $j>1$,
\begin{equation}
\frac{1}{\varepsilon}\partial_{\tau}T^{\mathrm{u},\mathrm{s}}_{j}(u,\tau)+\partial_uT^{\mathrm{u},\mathrm{s}}_{j}(u,\tau)=-\frac{1}{8}\cosh(u)^2\sum_{l=1}^{j-1}\partial_uT^{\mathrm{u},\mathrm{s}}_{l}(u,\tau)\cdot\partial_uT^{\mathrm{u},\mathrm{s}}_{j-l}(u,\tau),
\label{eq_outer_mu_2}
\end{equation}
with boundary condition (see \eqref{eq_init_cond_Fourier_mu})
\begin{equation}
|\partial_uT_j^{\mathrm{u},\mathrm{s}}(u,\tau)|\leq M e^{-2|\Re(u)|}.
\label{eq_boundary_condition_Fourier}
\end{equation}
We proceed by induction. 
We only deal with the unstable case as the stable case is analogous. 

Consider $j=1$. 
Expanding equation \eqref{eq_outer_mu_1} in Fourier series we obtain:

$$
\frac{ik}{\varepsilon}T_{1}^{\mathrm{u},[k]}(u)+\partial_uT_{1}^{\mathrm{u},[k]}(u)=\frac{2}{\cosh^2(u)}g^{[k]},\qquad k\in\mathbb{Z}.
$$
When $g^{[k]}=0$, the only solution is $T_{1}^{\mathrm{u},[k]}(u)=C\cdot e^{- iku/\varepsilon}$ but, since we impose \eqref{eq_boundary_condition_Fourier}, necessarily $C=0$. 
Thus, if $k\notin G_1$, $\partial_uT^{\mathrm{u},[k]}_{1}(u)=0$.

Now consider $j>1$ and assume that for $\nu=1,\dots,j-1$ if $\ell\notin G_\nu$, $\partial_u T_\nu^{\mathrm{u},[\ell]}=0$. 
Expanding \eqref{eq_outer_mu_2} in Fourier series:
$$
\frac{ik}{\varepsilon}T_{j}^{\mathrm{u},[k]}(u)+\partial_uT^{\mathrm{u},[k]}_{j}(u)=-\frac{1}{8}\cosh^2(u)\sum_{l=1}^{j-1}\sum_{m\in\mathbb{Z}}\partial_uT_{l}^{\mathrm{u},[m]}(u)\cdot\partial_uT_{j-l}^{\mathrm{u},[k-m]}(u).
$$
The non-zero terms on the right-hand side are those where $m\in G_l$ and $k-m\in G_{j-l}$. 
This means $k=m+(k-m)=m_1+\dots+m_l+m'_1+\dots+m'_{j-l}$ with $m_i,m'_i\in G_1$. 
This means $k\in G_j$. 
Therefore, if $k\notin G_j$ the right-hand side of the equation is $0$ and, imposing \eqref{eq_boundary_condition_Fourier}, $\partial_uT_{j}^{\mathrm{u},[k]}(u)=0$.
\end{proof}

\begin{remark} 
We emphasize that, using Lemma \ref{lemma_Fourier_coeff_outer}, we are also able to control the order in $\mu$ of the other harmonics of $\Delta^{[k]}(u,\mu)$. Indeed, if we define
$$
n_k(g):=\min\{ \ell \in \mathbb{N} : k \in G_\ell\},
$$
then $\Delta^{[k]}(u,\mu)= \mathcal{O}(\mu^{n_k})$. This fact could be useful for a further analysis in the degenerate case $\chi_n^{[-1]}=0$ but it is out of the scope of this work. 
\end{remark}

\subsection{Analysis of $\Delta^{<n}$}
\label{sec_error_lower_order_terms}

We consider the function $\Delta^{<n}$ defined by~\eqref{decompositionDelta_heuristics}:
$$
\Delta^{<n}(u,\tau,\mu)= \sum_{j<n} \Delta_j(u,\tau) \cdot \mu^j
$$
We prove the following proposition:
\begin{proposition}
For $j=1,\dots,n-1$ and $u\in\mathcal{D}_{\domvar,\phi}^{\mathrm{out}}\cap\mathbb{R}$, $\tau\in\mathbb{T}$ we have that:
$$
|\partial_u\Delta_j(u,\tau)|\leq M\cdot  \frac{e^{-2(\frac{\pi}{2\varepsilon}-2\domvar)}}{\varepsilon^2\domvar^4}.
$$
As a consequence $$|\partial_u \Delta^{<n}(u,\tau, \mu) | \leq M\cdot \frac{e^{-2(\frac{\pi}{2\varepsilon}-2\domvar)}}{\varepsilon^2\domvar^4}.$$
\label{prop_higher_exponential}
\end{proposition}

In order to prove Proposition \ref{prop_higher_exponential}, we first recall that, by Theorem \ref{teo_difference_manifolds},
\begin{equation}\label{doubleexpr_Delta}
\Delta(u,\tau,\mu)=\sum_{k\in \mathbb{Z}} \Upsilon^{[k]}(\mu)\cdot e^{i k (u/\varepsilon - \tau+ \mathcal{C}(u,\tau,\mu))} = \sum_{j\geq 0} \Delta_j(u,\tau) \cdot \mu^j
\end{equation}
with $\mathcal{C}(u,\tau,\mu)$ having the Taylor expansion $\mathcal{C}(u,\tau,\mu)= \sum_{j\geq 0 } \mathcal{C}_j (u,\tau) \cdot \mu^j$. We split the proof in three parts. 
First, in Lemma \ref{lemma_bound_upsilon_no_tilde}
we provide an exponentially small bound for $\Upsilon^{[k]}$. Then, in Lemma \ref{lem:explicit_formula_Deltajk} we express $\Delta^{[k]}_j$ (the $k$-Fourier coefficient of $\Delta_j$) in terms of $\Upsilon^{[m]}_l$ and $\mathcal{C}_\ell$. Finally, in Lemma \ref{lemma_bound_upsilon_k_1}, using Lemma \ref{lemma_Fourier_coeff_outer} too, we provide an improved bound for $\Upsilon^{[\pm 1]}_j$, $j=1,\cdots, n-1$. 
This allows us to finish the proof of Proposition \ref{prop_higher_exponential}.

\begin{lemma}
Take $\nu\in \mathbb{N}$. 
There exists a constant $M$ such that
for any $j=1,\cdots, \nu$ and $k\in \mathbb{Z}$, the Taylor coefficients $\Upsilon_j^{[k]}$ satisfy 
\begin{equation}
    |\Upsilon^{[k]}_j|\leq M\cdot\frac{e^{-|k|(\frac{\pi}{2\varepsilon}-\domvar-M/\domvar)}}{\varepsilon\domvar^3}\cdot e^{-|k|\sigma}.
    \label{eq_bound_upsilon}
\end{equation}
\label{lemma_bound_upsilon_no_tilde}
\end{lemma}

\begin{proof}
By Theorem \ref{teo_existence_outer}, if $(u,\tau)\in\mathcal{D}_{\domvar,\phi}^{\mathrm{out}}\times\mathbb{T}_\sigma$
$$
|
\partial_u\Delta(u,\tau,\mu)
|\leq |\partial_uT^{\mathrm{out},\mathrm{s}}(u,\tau,\mu)|+|\partial_uT^{\mathrm{out},\mathrm{u}}(u,\tau,\mu)|\leq  \frac{M}{\varepsilon^2\domvar^3}.
$$
We consider the change of variables $(w,\tau)=h(u,\tau)=(u+\varepsilon\cdot\C(u,\tau),\tau)
$. 
It is clearly well defined and injective, as
$$
\partial_u w(u,\tau)=1+\mathcal{O}\left(\frac{\mu}{\domvar}\right).
$$
We introduce 
$$
\widetilde{\Delta}(w,\tau,\mu)=\Delta(h^{-1}(w,\tau),\mu)=\sum_{k\in\mathbb{Z}}\Upsilon^{[k]}(\mu)\cdot e^{ik\left(\frac{w}{\varepsilon}-\tau\right)}
$$
and we have that:
\begin{equation}
|\partial_w\widetilde{\Delta}(w,\tau,\mu)|\leq\frac{M}{\varepsilon^2\domvar^3}.
\label{eq_bound_w}
\end{equation}
From \eqref{eq_bound_w} and using that $\partial_w\widetilde{\Delta}(w,\tau,\mu)$ is analytic in $\tau$ in a strip of width $\sigma$ we bound each Fourier coefficient:
$$
\left|\frac{ik}{\varepsilon}\Upsilon^{[k]}(\mu)\cdot e^{ikw/\varepsilon}\right|\leq M\cdot\frac{e^{-|k|\sigma}}{\varepsilon^2\domvar^3}.
$$
We use this inequality to obtain bounds for $\Upsilon^{[k]}$. 
We first take $k>0$ and we consider the point $u^*=-i(\pi/2-\varepsilon\domvar)$ and $\C^*=\C(u^*,\tau)$. We particularize the previous inequality —valid for all $w$— for the value $w^*=-i(\pi/2-\varepsilon\domvar)+\varepsilon\cdot\C^*$:
$$
\left|\frac{ik}{\varepsilon}\Upsilon^{[k]}(\mu,\varepsilon)\cdot e^{|k|(\frac{\pi}{2\varepsilon}-\domvar+i\C^*)}\right|\leq  M\cdot\frac{e^{-|k|\sigma}}{\varepsilon^2\domvar^3}.
$$
As $|\C^*|\leq  \frac{M}{\domvar}$,
$$|\Upsilon^{[k]}(\mu)|\leq M\cdot\frac{e^{-|k|(\frac{\pi}{2\varepsilon}-\domvar-M/\domvar)}}{\varepsilon\domvar^3}\cdot e^{-|k|\sigma}
$$
and, by Lemma \ref{lemma_bounds_coefficients_powerseries}, we get the result. 
For $k<0$ we argue analogously with $u^*=i(\pi/2-\varepsilon\domvar)$.
\end{proof}
In the following lemma we find an explicit formula for $\Delta_j^{[k]}(u)$.  

\begin{lemma}  \label{lem:explicit_formula_Deltajk} 
Take $j\geq 0$, $k\in \mathbb{Z}$. 
Then, for all $u\in \mathcal{D}^{\mathrm{out}}_{\domvar,\phi}$
the $k$-Fourier coefficient of $\Delta_j$ can be expressed as
\begin{equation}
    \Delta_j^{[k]}(u)=e^{-ik\frac{u}{\varepsilon}}\Upsilon_j^{[-k]}+\sum_{m\neq 0}e^{im\frac{u}{\varepsilon}}\left[\Upsilon_{j-1}^{[m]}\cdot\widetilde{\mathcal{C}}_{1,m}^{[m+k]}(u)+\dots+\Upsilon_1^{[m]}\cdot\widetilde{\mathcal{C}}_{j-1,m}^{[m+k]}(u)\right]
    \label{eq_formula_Fourier_coeff},
\end{equation}
with 
\begin{equation}
\widetilde{\mathcal{C}}_{j-\nu,k}(u,\tau)=\sum_{l=1}^{j-\nu}\frac{(ik)^l}{l!}\sum_{\substack {a_1+\dots +a_l=j-\nu\\a_m\geq 1}}\C_{a_1}(u,\tau)\dots \C_{a_l}(u,\tau).
\label{eq_def_c_tilde}
\end{equation}
\end{lemma}
\begin{proof}
We use that $\mathcal{C}(u,\tau,\mu)$ and $\Upsilon^{[k]}(\mu)$ depend analytically on $\mu$ and that $\Delta(u,\tau,0)=\mathcal{C}(u,\tau,0)=\Upsilon^{[k]}(0)=0$, so they admit the expansions $\mathcal{C}(u,\tau,\mu)=\sum_{j\geq 1}\mathcal{C}_j(u,\tau)\cdot \mu^j$ and $\Upsilon^{[k]}(\mu)=\sum_{j\geq 1}\Upsilon_j^{[k]}\cdot \mu^j$.
We fix $j\geq 1$ and we remark that $M$ denotes a generic constant that can (and usually will) depend on $j$. Using these expansions and equating the terms of order $\mu^j$ in the expression~\eqref{doubleexpr_Delta} for $\Delta(u,\tau,\mu)$, we obtain
\begin{equation}
\Delta_j(u,\tau)=\sum_{k\in\mathbb{Z}}e^{ik(u/\varepsilon-\tau)}\cdot\left[
\Upsilon_j^{[k]}+\sum_{\nu=1}^{j-1}\Upsilon_{\nu}^{[k]} \cdot\widetilde{\mathcal{C}}_{j-\nu,k}(u,\tau)
\right],
\label{eq_formula_order_mu}
\end{equation}
with $\widetilde{\mathcal{C}}_{j-\nu,k}(u,\tau)$ defined in~\eqref{eq_def_c_tilde}.
 We have used the absolute convergence of the series in $\mu$ of $\Upsilon^{[k]}(\mu)$ and $\mathcal{C}(u,\tau,\mu)$ to rearrange terms in the formula. To work with the series we need to find bounds for $\widetilde{\mathcal{C}}_{1,k}, \cdots , \widetilde{\mathcal{C}}_{j-1,k}$. We use that, from Theorem \ref{teo_difference_manifolds}, for $u\in\mathcal{D}^{\mathrm{out}}_{\domvar,\phi}$ and $\tau\in\mathbb{T}_\sigma$:
$$
|\C_{a_m}(u,\tau)|\leq \frac{M}{\domvar}
$$
and then
\begin{equation}
|\widetilde{\mathcal{C}}_{\ell,k}(u,\tau)|\leq M\cdot\frac{|k|^\ell}{\domvar^\ell}\qquad \ell=1,\dots,j-1.
\label{eq_bounds_C}
\end{equation}
Note that $\widetilde{\mathcal{C}}_{\ell,0}=0$.

We consider now the Fourier series of $\Delta_j$ and $\widetilde{\mathcal{C}}_{\ell,\domvar}$ (in \eqref{eq_def_c_tilde}): $
\Delta_j(u,\tau)=\sum_{k\in\mathbb{Z}}\Delta_j^{[k]}(u)\cdot e^{ik\tau}
$ and $\tilde{\C}_{\ell,k}(u,\tau)=\sum_{l\in\mathbb{Z}}\tilde{\C}_{\ell,k}^{[l]}(u)\cdot e^{il\tau}.
$
We plug this expression in~\eqref{eq_formula_order_mu} and we obtain
\begin{align}
\Delta_j(u,\tau)&=\sum_{k\in\mathbb{Z}}e^{ik(u/\varepsilon-\tau)}\left[
\Upsilon_j^{[k]}+\sum_{l\in\mathbb{Z}}e^{il\tau}\left(\Upsilon_{j-1}^{[k]}\cdot\widetilde{\mathcal{C}}_{1,k}^{[l]}(u)+\Upsilon_{j-2}^{[k]}\cdot\widetilde{\mathcal{C}}_{2,k}^{[l]}(u)+\dots+\Upsilon_1^{[k]}\cdot\widetilde{\mathcal{C}}_{j-1,k}^{[l]}(u)\right)
\right].
\label{eq_delta_j_partial_fourier}
\end{align}

Since $\tilde{\C}_{\nu,k}(u,\tau)$ is analytic in $\tau$ in a strip of width $\sigma$ and satisfies bound \eqref{eq_bounds_C} we have that:
\begin{equation}
|\tilde\C_{\nu,k}^{[l]}(u)|\leq M\cdot\frac{|k|^\nu}{\domvar}\cdot e^{-|l|\sigma}.
\label{eq_bound_C_fourier}
\end{equation}
where $M$ depends only on $j$. Besides, for $u\in\mathcal{D}_{\domvar,\phi}^{\mathrm{out}}$, $\tau\in\mathbb{T}_{\sigma/2}$
$$
\left|e^{ik(u/\varepsilon-\tau)}\right|\leq e^{|k|(\frac{\pi}{2\varepsilon}-\domvar+\sigma/2)}.
$$
With those bounds and Lemma \ref{lemma_bound_upsilon_no_tilde} we can check all the terms in \eqref{eq_delta_j_partial_fourier} are absolutely convergent series. Indeed, for $u\in\mathcal{D}_{\domvar,\phi}^{\mathrm{out}}$, $\tau\in\mathbb{T}_{\sigma/2}$, the first term
$$
\left|\sum_{k\in\mathbb{Z}}
e^{ik(u/\varepsilon-\tau)}\cdot\Upsilon_j^{[k]}\right|\leq \sum_{k\in\mathbb{Z}}e^{|k|(\frac{\pi}{2\varepsilon}-\domvar)}e^{|k|(\sigma/2)}\cdot\frac{M}{\varepsilon\domvar^3}\cdot e^{-|k|(\frac{\pi}{2\varepsilon}-\domvar-M/\domvar)}e^{-|k|\sigma}=M\cdot\sum_{k\in\mathbb{Z}} e^{-|k|(\sigma/2-M/\domvar)
}< \infty,
$$
where in the last inequality we have used that, for large enough $\domvar$, the exponent $-(\sigma/2-M/\domvar)$ is negative. As for the next terms, we take $\nu=1,\dots j-1$ and we have:
\begin{align*}
\left|\sum_{k\in\mathbb{Z}}\right.&\left.e^{ik(u/\varepsilon-\tau)}\cdot \Upsilon_{j-\nu}^{[k]}\sum_{l\in\mathbb{Z}}e^{il\tau}\cdot\tilde{\C}_{\nu,k}^{[l]}(u)\right|\leq\sum_{k\in\mathbb{Z}}e^{|k|(\frac{\pi}{2\varepsilon}-\domvar)}e^{|k|(\sigma/2)}\cdot\frac{M}{\varepsilon^2\domvar^3}\cdot e^{-|k|(\frac{\pi}{2\varepsilon}-\domvar-M/\domvar)}e^{-|k|\sigma}\cdot M\cdot\frac{|k|^\nu}{\domvar}\\&\leq \frac{M}{\varepsilon^2\domvar^{4}}\cdot\sum_{k\in\mathbb{Z}}e^{-|k|(\sigma/2-M/\domvar)}\cdot|k|^{\nu},
\end{align*}
where the last sum is finite provided, again, $\domvar$ is large enough.

Since all terms converge absolutely, we can rearrange expression \eqref{eq_delta_j_partial_fourier}:
{\small
$$
\Delta_j(u,\tau)=
\sum_{k\in\mathbb{Z}}e^{ik(u/\varepsilon-\tau)}\cdot
\Upsilon_j^{[k]}+\sum_{k\neq 0}\sum_{l\in\mathbb{Z}}e^{iku/\varepsilon}e^{i(l-k)\tau}\cdot\left[\Upsilon_{j-1}^{[k]}\cdot \widetilde{\mathcal{C}}^{[l]}_{1,k}(u)+\Upsilon_{j-2}^{[k]}\cdot\widetilde{\mathcal{C}}_{2,k}^{[l]}(u)+\dots+\Upsilon_1^{[k]}\cdot\widetilde{\mathcal{C}}_{j-1,k}^{[l]}(u)\right],
$$
}where we have used that (see \eqref{eq_def_c_tilde}) $\tilde{\C}_{1,0}=\dots=\tilde{\C}_{j-1,0}=0$. We perform some changes in the indices: first, we set $l-k=m$. Since $l$ runs over all integers, $m$ does too:
{\small
$$
\Delta_j(u,\tau)=\sum_{k}e^{ik\frac{u}{\varepsilon}}\cdot
\Upsilon_j^{[k]}\cdot e^{-ik\tau}+\sum_{k\neq 0}\sum_{m\in\mathbb{Z}}e^{ik\frac{u}{\varepsilon}}e^{im\tau}\cdot\left[\Upsilon_{j-1}^{[k]}\cdot \widetilde{\mathcal{C}}_{1,k}^{[m+k]}(u)+\Upsilon_{j-2}^{[k]}\cdot\widetilde{\mathcal{C}}_{2,k}^{[m+k]}(u)+\dots+\Upsilon_1^{[k]}\cdot\widetilde{\mathcal{C}}_{j-1,k}^{[m+k]}(u)
\right]
$$
} and the expression~\eqref{eq_formula_Fourier_coeff} for $\Delta_j^{[k]}(u)$ is proven for 
$u\in\mathcal{D}^{\mathrm{out}}_{\domvar,\phi}$.
\end{proof}

We recall that, by definition, $n\in\mathbb{N}$ is such that $1\in G_n$ and $1\notin G_j$ for $j=1,\dots,n-1$. Then, Lemma \ref{lemma_Fourier_coeff_outer} implies that $\Delta_j^{[\pm 1]}(u)=0$ for $j=1,\dots,n-1$. In the next lemma we will use this fact as a condition for a sharper bound on the coefficients $\Upsilon_j^{[\pm 1]}$. 

\begin{lemma}\label{lemma_bound_upsilon_k_1}
For $j=1,\dots,n-1$ the following bound for the Taylor coefficient $\Upsilon_j^{[\pm 1]}$ holds:
\begin{equation}
|\Upsilon_j^{[\pm 1]}|\leq M\cdot\frac{e^{-2(\frac{\pi}{2\varepsilon}-\domvar-M/\domvar)}}{\varepsilon\domvar^{4}},
\label{eq_improved_bound_Upsilon}
\end{equation}
for $\domvar$ as defined in formula \eqref{eq_definition_outer_domains} large enough.
\end{lemma}
\begin{proof}

Assume $n>1$ (the case $n=1$ is void). When $j=1$, by formula \eqref{eq_formula_Fourier_coeff}:
$$
\Delta_1^{[\pm 1]}(u)=e^{\mp i\frac{u}{\varepsilon}}\cdot \Upsilon_1^{[\mp 1]}.
$$
Since $1\notin G_1$, Lemma \ref{lemma_Fourier_coeff_outer} implies $\Delta_1^{[\pm 1]}(u)=0$ and hence, $\Upsilon_1^{[\mp 1]}=0$.
In particular, it satisfies the inequality in the statement.

Take $j=2,\dots,n-1$. Assume by induction that $\Upsilon_\nu^{[\pm 1]}$ satisfies the bound in \eqref{eq_improved_bound_Upsilon} if $\nu=1,\dots,j-1$.
Using formula \eqref{eq_formula_Fourier_coeff}:
$$
\Delta_j^{[k]}(u)=e^{-ik\frac{u}{\varepsilon}}\cdot\Upsilon_j^{[-k]}+\sum_{m\neq 0}e^{im\frac{u}{\varepsilon}}\cdot\left[\Upsilon_{j-1}^{[m]}\cdot\widetilde{\mathcal{C}}_{1,m}^{[k+m]}(u)+\dots+\Upsilon_1^{[m]}\cdot\widetilde{\mathcal{C}}_{j-1,m}^{[k+m]}(u)\right].
$$
Since $1\notin G_j$, by \ref{lemma_Fourier_coeff_outer}, $\Delta^{[\pm 1]}_j(u)
=0$. We take $k=1$ and equate the previous formula to $0$. Redistributing and replacing $u=0$:
\begin{align*}
\Upsilon_j^{[-1]}&=-\sum_{m\neq 0}\Upsilon_{j-1}^{[m]}\cdot\widetilde{\mathcal{C}}_{1,m}^{[1+m]}(0)+\dots+\Upsilon_1^{[m]}\cdot \widetilde{\mathcal{C}}_{j-1,m}^{[1+m]}(0)=\\
&=-\left(\Upsilon_{j-1}^{[1]}\cdot\widetilde{\mathcal{C}}_{1,1}^{[2]}(0)+\dots+\Upsilon_1^{[1]}\cdot \widetilde{\mathcal{C}}_{j-1,1}^{[2]}(0)+\Upsilon_{j-1}^{[-1]}\cdot\widetilde{\mathcal{C}}_{1,-1}^{[0]}(0)+\dots+\Upsilon_1^{[-1]}\cdot \widetilde{\mathcal{C}}_{j-1,-1}^{[0]}(0)\right)
\\
&-\left(\sum_{|m|>1}\Upsilon_{j-1}^{[m]}\cdot\widetilde{\mathcal{C}}_{1,m}^{[1+m]}(0)+\dots+\Upsilon_1^{[m]}\cdot \widetilde{\mathcal{C}}_{j-1,m}^{[1+m]}(0)\right)=:A+B.
\end{align*}
To bound $A$ we use the induction hypothesis along with the bounds of $\tilde{\C}_{\nu,m}^{[l]}$ given by \eqref{eq_bound_C_fourier}:
\begin{align*}
|A|&\leq 
M\cdot\frac{e^{-2(\frac{\pi}{2\varepsilon}-\domvar-M/\domvar)}}{\varepsilon\domvar^4} 
\cdot \left(\frac{M}{\domvar}\cdot e^{-2\sigma}(j-1)  +\frac{M}{\domvar}\cdot(j-1)\right)\leq M\cdot\frac{e^{-2(\frac{\pi}{2\varepsilon}-\domvar-M/\domvar)}}{\varepsilon\domvar^5}.
\end{align*}
To bound $B$ we use \ref{lemma_bound_upsilon_no_tilde} as well as \eqref{eq_bound_C_fourier}:
\begin{align*}
|B|\leq &\sum_{|m|>1}  \frac{M}{\varepsilon\domvar^3}\cdot e^{-|m|(\frac{\pi}{2\varepsilon}-\domvar-M/\domvar)}\cdot e^{-|m|\sigma}\left(M\cdot\frac{|m|}{\domvar}e^{-|1+m|\sigma}+\dots+M\cdot\frac{|m|^{j-1}}{\domvar}e^{-|1+m|\sigma}\right)\\
&\leq 
\sum_{|m|>1}  \frac{M}{\varepsilon\domvar^3}\cdot e^{-|m|(\frac{\pi}{2\varepsilon}-\domvar-M/\domvar)}\cdot e^{-|m|\sigma}M\cdot\frac{(j-1)}{\domvar}|m|^{j-1}e^{-\sigma}e^{-|m|\sigma}\\
&
\leq\sum_{|m|>1}\frac{M}{\varepsilon\domvar^4}\cdot |m|^{j-1}e^{-|m|(\frac{\pi}{2\varepsilon}-\domvar-M/\domvar+2\sigma)}\leq 
M\cdot\frac{e^{-2(\frac{\pi}{2\varepsilon}-\domvar-M/\domvar)}}{\varepsilon\domvar^4}.
\end{align*}
We complete the proof by combining the bounds for $A$ and $B$.

\end{proof}

By \ref{lemma_bound_upsilon_no_tilde} and \ref{lemma_bound_upsilon_k_1}, $\Upsilon_j^{[k]}$ with $j=1,\dots,n-1$ are, at least, of squared exponentially small order. To finish the proof of \ref{prop_higher_exponential} we only need to prove that the size of the coefficients transfers to the size of the function when $u$ and $\tau$ are real.

\begin{proof}[End of the proof of
\ref{prop_higher_exponential}]
We consider formula \eqref{eq_formula_order_mu} 
\begin{equation*}
\Delta_j(u,\tau)=\sum_{k\in\mathbb{Z}}e^{ik(u/\varepsilon-\tau)}\cdot\left[
\Upsilon_j^{[k]}+\sum_{\nu=1}^{j-1}\Upsilon_{\nu}^{[k]} \cdot\widetilde{\mathcal{C}}_{j-\nu,k}(u,\tau)
\right].
\end{equation*}
and evaluate it for $u\in\mathcal{D}^{\mathrm{out}}_{\domvar,\phi}\cap\{|\Im(u)|<\domvar\varepsilon\}$ and $\tau\in\mathbb{T}$.
We also use \eqref{eq_bounds_C} to bound $|\widetilde{\mathcal{C}}_{\nu,k}(u,\tau)|$. We split the sum into $k=\pm 1$ (we bound with Lemma \ref{lemma_bound_upsilon_k_1}) and $|k|>1$ (we use Lemma \ref{lemma_bound_upsilon_no_tilde}):
\begin{align*}
|\Delta_j(u,\tau)-&\Upsilon_j^{[0]}|=\left|\sum_{k\neq 0}e^{ik(u/\varepsilon-\tau)}\cdot\left[
\Upsilon_j^{[k]}+\sum_{\nu=1}^{j-1}\Upsilon_{\nu}^{[k]} \cdot\widetilde{\mathcal{C}}_{j-\nu,k}(u,\tau)
\right]\right|\leq\\
&
\sum_{k\neq 0}e^{|k|\frac{\domvar\varepsilon}{\varepsilon}}\cdot\left[
\left|\Upsilon_j^{[k]}\right|+\sum_{\nu=1}^{j-1}\left|\Upsilon_{\nu}^{[k]}\right| \cdot\left|\widetilde{\mathcal{C}}_{j-\nu,k}(u,\tau)\right|
\right]\leq\\
& e^\domvar\cdot|\Upsilon_j^{[1]}|+\sum_{\nu=1}^{j-1}e^\domvar\cdot\left|\Upsilon_{\nu}^{[1]}\right| \cdot\left|\widetilde{\mathcal{C}}_{j-\nu,1}(u,\tau)\right|+e^\domvar\cdot|\Upsilon_j^{[-1]}|+\sum_{\nu=1}^{j-1}e^\domvar\cdot\left|\Upsilon_{\nu}^{[-1]}\right| \cdot\left|\widetilde{\mathcal{C}}_{j-\nu,-1}(u,\tau)\right|\\
&+\sum_{|k|>1} e^{|k|\cdot\domvar}\cdot|\Upsilon_j^{[k]}|+\sum_{\nu=1}^{j-1}e^{|k|\cdot\domvar}\cdot\left|\Upsilon_{\nu}^{[k]}\right| \cdot\left|\widetilde{\mathcal{C}}_{j-\nu,k}(u,\tau)\right|\leq\\
&
M\cdot\frac{e^{-2(\frac{\pi}{2\varepsilon}-\domvar-M/\domvar)}}{\varepsilon\domvar^4}\cdot e^{\domvar}\cdot\left(1+\frac{j-1}{\domvar}\right)+\sum_{|k|>1}M\cdot\frac{e^{-|k|(\frac{\pi}{2\varepsilon}-\domvar-M/\domvar)}}{\varepsilon\domvar^3}\cdot e^{|k|\cdot\domvar}\cdot\left(1+\sum_{l=1}^{j-1}\frac{|k|^l}{\domvar}\right)\leq\\& 
M\cdot\frac{e^{-2(\frac{\pi}{2\varepsilon}-\frac{3}{2}\domvar-M/\domvar)}}{\varepsilon\domvar^4}+M\cdot\frac{e^{-2(\frac{\pi}{2\varepsilon}-2\domvar-M/\domvar)}}{\varepsilon\domvar^3}\leq M\cdot\frac{e^{-2(\frac{\pi}{2\varepsilon}-2\domvar)}}{\varepsilon\domvar^3}.
\end{align*}
Since the previous bounds hold for $u\in\mathcal{D}^{\mathrm{out}}_{\domvar,\phi}\cap\{|\Im(u)|< \domvar\varepsilon\}$, we obtain 
$$
|\partial_u\Delta_j(u,\tau)|\leq M\cdot\frac{e^{-2(\frac{\pi}{2\varepsilon}-2\domvar)}}{\varepsilon^2\domvar^4}
$$
via a Cauchy estimate. Note that $\partial_u\Upsilon_j^{[0]}=0$.
\end{proof}

\subsection{Analysis of $\Delta^{\geq n}$} \label{sec_error_estimate_higher_order_terms}

Define $\delta_0$

\begin{equation}
\delta_{0}(u,\tau,\mu):=\frac{2e^{-\frac{\pi}{2\varepsilon}}}{\varepsilon} \cdot\Re \left(\chi^{[-1]}(\mu)\cdot e^{i\left(\tau-\frac{u}{\varepsilon}\right)}\right).
\label{eq_reminder_delta_0}
\end{equation}
From~\eqref{eq_known_splitting_inner} in Theorem \ref{teo_difference_manifolds} we know that 
\begin{equation}\label{first_bound_Delta>n}
\left | \partial_u \Delta (u,\tau,\mu) - \partial_u \delta_0(u,\tau,\mu) \right | \leq M\cdot \frac{|\mu|}{\log(1/\varepsilon)\cdot \varepsilon^2}\cdot e^{-\frac{\pi}{2\varepsilon}}.
\end{equation}

In this section we focus on the analysis of $\partial_u \Delta^{\geq n} - \partial_u \delta_0$,  with $\Delta^{\geq n}(u,\tau,\mu)= \sum_{j\geq n} \Delta_j(u,\tau) \cdot \mu^j $, the tail of the Taylor series of $\Delta$ around $\mu=0$, starting at $n$. 
To this end we use~\eqref{first_bound_Delta>n} together with a suitable version of Schwartz's lemma:

\begin{lemma} Let $\eta>0$, $\nu \in \mathbb{N}$ and $h$ be an analytic function of $\mu$ defined in $B_{\eta} \subset\mathbb{C}$. 
Assume that $\sup \{ |h(\mu)| : \mu \in B_\eta \} \leq M_h$ for some constant $M_h$.  
Let $h(\mu)=\sum_{j\geq 0}h_j\cdot\mu^j$ be its power expansion around $\mu=0$.
\begin{enumerate}
\item 
If $h^{(j)}(0)=0$ for $j=0,\dots,\nu-1$. Then $
|h(\mu)|\leq |\mu|^\nu\cdot\eta^{-\nu}\cdot M_h.
$
\item 
There exists a constant $M_2$ (depending only on $\nu$ and $\eta$) such that the function $h^{\geq \nu} (\mu) = \sum_{j\geq \nu} h_j\cdot\mu^j$ is bounded by $
|h^{\geq \nu} (\mu)|\leq |\mu|^\nu\cdot M_2\cdot M_h.
$
\end{enumerate}
\label{lemma_bound_tail}
\end{lemma}

\begin{remark}
We will be using Lemma \ref{lemma_bound_tail} for functions depending on $u,\tau,\varepsilon$ and $\mu$. 
We will consider the analytic dependence on $\mu$ and regard the rest of the variables as parameters. 
Note that the constants appearing in the bounds of the Lemma only depend on the radius of the ball of analyticity with respect to $\mu$ and the integer $\nu$. 
In particular, the dependence on the bounds on the parameter $\varepsilon$ remains unaltered.
\end{remark}

The following proposition is an almost straightforward consequence of bound~\eqref{first_bound_Delta>n} and Lemma \ref{lemma_bound_tail}.  
\begin{proposition}
Let $ {\Delta}^{\geq n} (u,\tau,\mu)$ be the tail of the Taylor series of $\Delta(u,\tau,\mu)$. Then
\begin{equation}
|\partial_u\ {\Delta}^{\geq n}(u,\tau,\mu)-\partial_u\delta_0(u,\tau,\mu)|\leq M\cdot\frac{|\mu|^n}{\log(1/\varepsilon)\cdot \varepsilon^2}\cdot e^{-\frac{\pi}{2\varepsilon}}.
\label{eq_bound_tail}
\end{equation}
\label{prop_bound_tail}
\end{proposition}

\begin{proof} 
From~\eqref{first_bound_Delta>n} and taking into account that $|\mu|\leq \mu_0$, we have that 
$$
|\partial_u\Delta(u,\tau,\mu)-\partial_u\delta_0(u,\tau,\mu)|\leq M\cdot \frac{e^{-\frac{\pi}{2\varepsilon}}}{ \log(1/\varepsilon)\cdot\varepsilon^2}.
$$

Moreover, by Theorem \ref{teo_main_second},  $\chi^{[-1]}(\mu)=\sum_{j\geq n}\chi^{[-1]}_j\cdot\mu^j$. Since the only dependence on $\mu$ of $\delta_0$ is through $\chi^{[-1]}(\mu)$, $\delta_0(u,\tau,\mu)= {\delta}_0^{\geq n}(u,\tau,\mu)$. 
Then  using Lemma \ref{lemma_bound_tail} with $\nu=n$ we obtain the bound

\begin{equation*}
|\partial_u \Delta^{\geq n }(u,\tau,\mu) - \partial_u \delta_0^{\geq n }(u,\tau,\mu)| =|\partial_u\ {\Delta}^{\geq n}(u,\tau,\mu)- \partial_u\delta_0(u,\tau,\mu)|\leq M\cdot \frac{|\mu|^n}{\log(1/\varepsilon)\cdot \varepsilon^2}\cdot e^{-\frac{\pi}{2\cdot \varepsilon}}.
\end{equation*}
\vspace{0.1cm}
\end{proof}

\subsection{End of the proof of Theorem \ref{thm_main}}
\label{sec_completion_of_the_proof}

To obtain the first asymptotic expression we differentiate \eqref{eq_reminder_delta_0}:

$$
\partial_u\delta_{0}(u,\tau,\mu)=\frac{2e^{-\frac{\pi}{2\varepsilon}}}{\varepsilon^2} \cdot\Im\left(\chi^{[-1]}(\mu)\cdot e^{i\left(\tau-\frac{u}{\varepsilon}\right)}\right).
$$
Let us now perform the error estimates. 
Taking into account that $\partial_u\delta_0=\partial_u\delta_0^{\geq n}$, we split the error term:
$$
\partial_u\Delta-\partial_u\delta_0=\left(\sum_{j=1}^{n-1}\partial_u\Delta_j\cdot \mu^j\right)+\partial_u\Delta^{\geq n}-\partial_u\delta_0.
$$
For the first part we use Proposition \ref{prop_higher_exponential}

Then

$$
\left|\sum_{j=1}^{n-1}\partial_u\Delta_j\cdot\mu^j\right|\leq\sum_{j=1}^{n-1}|\partial_u\Delta_j|\cdot|\mu^j|\leq \sum_{j=1}^{n-1}M\cdot\frac{e^{-2(\frac{\pi}{2\varepsilon}-2\domvar)}}{\varepsilon^2\domvar^4}\cdot|\mu|^j\leq M\cdot\frac{|\mu|}{\varepsilon^{2}}\cdot e^{-\frac{\pi}{2\varepsilon}\cdot 2}.
$$

To bound $|\partial_u\Delta^{\geq n}-\partial_u\delta_0|$ we use Proposition \ref{prop_bound_tail}. Then,

$$
|\partial_u\Delta(u,\tau,\mu)-\partial_u\delta_0(u,\tau,\mu)|\leq M\cdot\frac{|\mu|}{\varepsilon^{2}}\cdot e^{-\frac{\pi}{2\varepsilon}\cdot 2}+M\cdot \frac{|\mu|^n}{\log(1/\varepsilon)\cdot \varepsilon^2}\cdot e^{-\frac{\pi}{2\cdot \varepsilon}},
$$
whence the result follows.

As for item 3, we deduce it by taking the particular cases $n=1$ and $n=2$ in Theorem \ref{teo_main_second}.

\newpage
\section{Acknowledgements}
This work is part of the grants PGC2018-098676-B-100 and PID-2021-122954NB-100 funded by 
MCIN/AEI/10.13039/\\501100011033 and “ERDF A way of making Europe.” T. M. S. is supported by the Catalan Institution for Research and Advanced Studies via an ICREA Academia Prize 2019. R.M. is supported by grant PRE2019-088132 funded by MCIN/AEI/10.13039/501100011033 and “ESF Investing in your future”. This work is also supported by the Spanish State Research Agency, through the Severo Ochoa and María de Maeztu Program for Centers and Units of Excellence in R\& D (CEX2020-001084-M).

\appendix

\section{The function $\mathcal{C}$. Proof of Theorem \ref{teo_difference_manifolds}}
\label{sec_appendix_proof_c}

For the proof of this theorem we adapt the methodology in \cite{bib_exp_small}: we split the equation in two parts and use a fixed point argument to find the solution in a specific function space. We work with functions defined in $\mathcal{D}^{\mathrm{out}}_{\domvar,\phi}\times\mathbb{T}_\sigma$ (see \eqref{eq_definition_outer_domains}) and we define the Fourier norm:
$$
||f||_{\alpha,\sigma}=\sum_{k\in\mathbb{Z}}||f^{[k]}||_{\alpha}\cdot e^{|k|\sigma}.
$$
where
$$
||f||_{\alpha}=\sup_{u\in\mathcal{D}^{\mathrm{out}}_{\domvar,\phi}}\{|u^2+(\pi/2)^2|^\alpha\cdot |f(u)|\}.
$$
We consider the following Banach spaces:
$$
\mathcal{P}_{\alpha}=\{f(u,\tau):\mathcal{D}^{\mathrm{out}}_{\domvar,\phi}\times\mathbb{T}_\sigma\to\mathbb{C},\, f\text{ analytic and } ||f||_{\alpha,\sigma}<\infty\}.
$$

The function $\mathcal{C}$ is such that $\Delta(u,\tau,\mu)=\Upsilon(u-\varepsilon\tau+\varepsilon\cdot\mathcal{C}(u,\tau,\mu))$. By subtracting the Hamilton-Jacobi equation \eqref{eq_HJ} for the stable and unstable manifolds we obtain:
$$
\frac{1}{\varepsilon}\partial_{\tau}\C(u,\tau,\mu)+\partial_u\C(u,\tau,\mu)=-\frac{1}{8}\cosh^2(u)(\partial_uT^\mathrm{u}(u,\tau,\mu)+\partial_uT^\mathrm{s}(u,\tau,\mu))\left(\frac{1}{\varepsilon} +\partial_u\C(u,\tau,\mu)\right).
$$
Denoting $A(u,\tau,\mu)=-\frac{1}{8}\cosh^2(u)(\partial_uT^\mathrm{u}(u,\tau,\mu)+\partial_uT^\mathrm{s}(u,\tau,\mu))$ we rewrite the equation as:
$$
\frac{1}{\varepsilon}\partial_{\tau}\C(u,\tau,\mu)+\partial_u\C(u,\tau,\mu)=\frac{1}{\varepsilon}A(u,\tau,\mu)+A(u,\tau,\mu)\partial_u\C(u,\tau,\mu)
$$
or, equivalently,
\begin{equation}
\frac{1}{\varepsilon}\partial_{\tau}\C(u,\tau,\mu)+\partial_u\C(u,\tau,\mu)=\frac{1}{\varepsilon}A(u,\tau,\mu)+\partial_u(A(u,\tau,\mu)\C(u,\tau,\mu))-\C(u,\tau,\mu)\partial_uA(u,\tau,\mu).
\label{eq_appendix_C}
\end{equation}
This equation is of the type 
$$
\frac{1}{\varepsilon}\partial_{\tau}\C(u,\tau)+\partial_u\C(u,\tau)=h(u,\tau).
$$
In order to invert the linear operator on the left-hand side in the domain $\mathcal{D}_{\domvar,\phi}^{\mathrm{out}}$ we expand it in Fourier series and define an inverse for each harmonic:
$$
\C^{[k]}(u)=\C^{[k]}(a_k)\cdot e^{\frac{ik}{\varepsilon}(a_k-u)}+\int_{a_k}^{u}e^{i\frac{k}{\varepsilon}(s-u)} h^{[k]}(s)ds,
$$
where $a_k=i\left(\frac{\pi}{2}-\domvar\varepsilon\right)$ if $k>0$, $a_k=-i\left(\frac{\pi}{2}-\domvar\varepsilon\right)$ if $k<0$ and $a_0=-\rho$ (see (277) in \cite{bib_exp_small}).
Since we are looking for any solution, we pick $\C^{[k]}(a_k)=0$ and we have:
$$
\C^{[k]}(u)=\int_{a_k}^{u}e^{i\frac{k}{\varepsilon} (s-u)} h^{[k]}(s)ds.
$$
Using the notation
\begin{equation}
\G^{[k]}(h)=\int_{a_k}^{u}e^{i\frac{k}{\varepsilon} (s-u)} h^{[k]}(s)ds
\end{equation}
we can define the inverse operator as:
\begin{equation}
\G(h)=\sum_{k\in\mathbb{Z}}\K{\G}(h)e^{ik\tau}.
\label{eq_operator_G_proof}
\end{equation}
We state in the following lemma the relevant properties of $\mathcal{G}$ and refer for the details of the proof to Lemma 9.2 in \cite{bib_exp_small}.

\begin{lemma}
\label{lemma_properties_G_C}
The operator $\G$ defined on $\mathcal{P}_{\alpha}$ satisfies the following properties. For $h\in\mathcal{P}_\alpha$ with $\alpha\geq 0$.

\begin{enumerate}
\item $\partial_u(\G(h))\in\mathcal{P}_{\alpha}$ and $||\partial_u(\G(h))||_{\alpha,\sigma}\leq M\cdot||h||_{\alpha,\sigma}$.

\item \textit{If} $h^{[0]}=0$, $\G(h)\in\mathcal{P}_\alpha$ and $||\G(h)||_{\alpha,\sigma}\leq M \cdot\varepsilon||h||_{\alpha,\sigma}$.

\item $\G(\partial_u h)\in \mathcal{P}_{\alpha}$ and $||\G(\partial_u h)||_{\alpha,\sigma}\leq M\cdot||h||_{\alpha,\sigma}$.

\item \textit{If} $\alpha >1$ $\G(h)\in\mathcal{P}_{\alpha-1}$ and $||\G(h)||_{\alpha-1,\sigma}\leq M\cdot||h||_{\alpha,\sigma}$.

\item $\G(h)\in\mathcal{P}_\alpha$ and $||\G(h)||_{\alpha,\sigma}\leq M\cdot||h||_{\alpha,\sigma}$.
\end{enumerate}
\end{lemma}
In the following lemma we state and prove some properties of the function $A(u,\tau,\mu)$.

\begin{lemma}
\label{lemma_prop_A}
The function $A(u,\tau,\mu)$ satisfies  $A\in\mathcal{P}_1$, $\partial_u A\in\mathcal{P}_1$, $A^{[0]}\in\mathcal{P}_2$ and $||A(u,\tau,\mu)||_{1,\sigma}\leq M\cdot|\mu|\varepsilon$, $||\partial_u A(u,\tau,\mu)||_{1,\sigma}\leq M\cdot\frac{|\mu|}{\domvar}$, $||A^{[0]}||_{2,\sigma}\leq M\cdot |\mu|^2\varepsilon^2$. As a consequence, $\left|\left|\frac{1}{\varepsilon}\G(A)\right|\right|_{1,\sigma}\leq M\cdot|\mu|\varepsilon$.
\end{lemma}

\begin{proof}
    By Theorem \ref{teo_existence_outer}, $\partial_uT^{s,u}\in \mathcal{P}_3$ and $||\partial_uT^{\mathrm{u},\mathrm{s}}||_{3,\sigma}\leq M |\mu|\varepsilon$. Besides, from the fact that $\cosh(u)$ has a pole of order 1 close to the singularities $\pm i\frac{\pi}{2}$ it follows that  $-\frac{1}{8}\cosh^2(u)\in\mathcal{P}_{-2}$. Thus, $A\in\mathcal{P}_1$ and
    $$
    ||A(u,\tau,\mu)||_{1,\sigma}\leq M\cdot |\mu|\varepsilon.
    $$
Due to the geometry of the domain and using Cauchy's formula for the derivative, we can find a bound for the derivative in the same space —reducing slightly $\domvar$ and $\rho$— dividing the norm by $\domvar\varepsilon$. This yields the bound
$$
||A(u,\tau,\mu)||_{1,\sigma}\leq M\cdot\frac{|\mu|}{\domvar}.
$$
As for the average, $A^{[0]}(u,\mu)$, we express it in terms of the average of the invariant manifolds:
$$
A^{[0]}(u,\mu) =-\frac{1}{8}\cosh^2(u)\left(\partial_uT^{\mathrm{u},[0]}(u,\mu)+\partial_uT^{\mathrm{s},[0]}(u,\mu)\right),
$$
As $T^{\mathrm{u}}(u,\tau,\mu)$ and $T^{\mathrm{s}}(u,\tau,\mu)$ satisfy equation \eqref{eq_HJ},
$$\frac{1}{\varepsilon}\partial_{\tau}T^{\mathrm{u},\mathrm{s}}(u,\tau,\mu)+\partial_uT^{\mathrm{u},\mathrm{s}}(u,\tau,\mu)=-\frac{1}{8}\cosh^2(u)(\partial_uT^{\mathrm{u},\mathrm{s}})^2(u,\tau,\mu)+2\mu \frac{g(\tau)}{\cosh^2(u
)},
$$
and, since $g^{[0]}=0$, we have:
$$
\partial_u{T^{\mathrm{u},\mathrm{s},[0]}}(u,\mu)=-\frac{1}{8}\cosh^2(u)((\partial_u{T^{\mathrm{u},\mathrm{s}}})^2)^{[0]}(u,\mu).
$$
From Theorem \ref{teo_existence_outer} we know $||\partial_u T^{\mathrm{u},\mathrm{s}}||_{3,\sigma}\leq M|\mu|\varepsilon$. Therefore, by property 5 of Lemma \ref{lemma_properties_G_C}
$$
||T^{\mathrm{u},\mathrm{s},[0]}||_{4}\leq M\cdot|\mu|^2\varepsilon^2.
$$
As a consequence,
$$
||A^{[0]}||_2\leq M\cdot|\mu|^2\varepsilon^2.
$$
Finally, we deal with $\frac{1}{\varepsilon}\G(A(u,\tau,\mu))$. We rewrite it as
$$
\frac{1}{\varepsilon}\G(A)=\frac{1}{\varepsilon}\G(A^{[0]})+\frac{1}{\varepsilon}\G(A-A^{[0]})=:N_1+N_2.
$$
By item 4 of Lemma \ref{lemma_properties_G_C}, $$||N_1||_{1,\sigma}\leq\frac{M}{\varepsilon}\cdot||A^{[0]}||_{2,\sigma}\leq M\cdot|\mu|^2\varepsilon$$ and, using item 2 of Lemma \ref{lemma_properties_G_C},
$$
||N_2||_{1,\sigma}\leq\frac{1}{\varepsilon}M\cdot\varepsilon||A-A^{[0]}||_{1,\sigma}\leq M\cdot|\mu|\varepsilon.
$$
\end{proof}

We now define the linear operator $
\mathcal{L}(h)=\G(\partial_u(A\cdot h))-\G(\partial_u A
\cdot h)$.
\begin{lemma} The operator
$\mathcal{L}:\mathcal{P}_1\longrightarrow \mathcal{P}_1$ is well defined and it satisfies $||\mathcal{L}(h)||_{1,\sigma}\leq \frac{M}{\domvar}\cdot||h||_{1,\sigma}$.
\label{lemma_appendix_l}
\end{lemma}
\begin{proof}
By Lemma \ref{lemma_prop_A}, $||A||_{1,\sigma}\leq M\cdot|\mu|\varepsilon$ and, by item 3 of Lemma \ref{lemma_properties_G_C},
$$
||\G(\partial_u(A\cdot h))||_{2,\sigma}\leq M||A\cdot h||_{2,\sigma}\leq M\cdot|\mu|\varepsilon||h||_{1,\sigma}.
$$
Therefore, 
$$
||\G(\partial_u(A\cdot h))||_{1,\sigma}\leq M\cdot\frac{|\mu|}{\domvar}||h||_{1,\sigma}.
$$
Using Lemma \ref{lemma_prop_A}, $||\partial_u A||_1\leq M\cdot\frac{|\mu|}{\domvar}$ and, by item 1 of Lemma \ref{lemma_properties_G_C},
$$
||\G(\partial_u A\cdot h)||_{1,\sigma}\leq M\cdot||\partial_u A \cdot h||_{2,\sigma}\leq M\cdot \frac{|\mu|}{\domvar}||h||_{1,\sigma}.
$$
\end{proof}
Finally, we can write equation \eqref{eq_appendix_C} as:
$$
(I-\mathcal{L})(\mathcal{C}(u,\tau,\mu))=\frac{1}{\varepsilon}\mathcal{G}(A(u,\tau,\mu)).
$$
By Lemma \ref{lemma_appendix_l}, $I-\mathcal{L}$ is invertible in $\mathcal{P}_1$, so that $\mathcal{C}(u,\tau,\mu)=(I-\mathcal{L})^{-1}\left(\frac{1}{\varepsilon}\G(A(u,\tau,\mu))\right)\in\mathcal{P}_1$ and, using Lemma \ref{lemma_prop_A},
$$
||\mathcal{C}(u,\mu,\tau)||_{1,\sigma}\leq M\cdot\left|\left|\frac{1}{\varepsilon}\G(A)\right|\right|_{1,\sigma}\leq M\cdot|\mu|\varepsilon.
$$
 We obtain the bound for the derivative by a straightforward application of Cauchy's formula and by reducing slightly $\domvar$ and $\rho$:
 $$
 ||\partial_u\mathcal{C}(u,\tau,\mu)||_{1,\sigma}\leq M\cdot\frac{|\mu|}{\domvar}.
 $$
\section{Proof of Lemma \ref{lemma_bound_tail} }

We begin with the first item. The function
$$
\tilde{h}(\mu)=\left\{
\begin{array}{ccl}
\frac{h(\mu)}{\mu^\nu}&& \mu\neq 0\\
\frac{h^{(\nu)}}{\nu!}&& \mu=0
\end{array}\right.
$$
is analytic in $\mathcal{B}_\eta(0)$.
The maximum principle forces the maximum of the function to be at a point $\mu^*$ such that $|\mu^*|=\eta$. Then
$$
|\tilde{h}(\mu)|=\left|\frac{h(\mu)}{\mu^\nu}\right|\leq\max_{|\mu|=\eta}\frac{|h(\mu)|}{|\mu^\nu|}\leq\frac{M_h}{\eta^\nu}
$$
and from here we have the result
$$
|h(\mu)|\leq \eta^{-\nu}\cdot|\mu|^\nu\cdot M_h.
$$
 
Now we proof the second item. 
Since
$$
h^{\geq\nu}(\mu)=h(\mu)-\sum_{j=0}^{\nu-1}h_j\cdot\mu^j,
$$
we use Lemma \ref{lemma_bounds_coefficients_powerseries} to bound the second term on the right-hand side. We have a constant $M_1$ depending only on $\nu$ such that $
|h_j|\leq M_1\cdot M_h$. Hence,
$$
\left|h^{\geq n}(\mu)\right|\leq |h(\mu)|+\left|\sum_{j=0}^{\nu-1}h_j\cdot \mu^j\right|\leq M_h+M_1\cdot M_h\cdot\sum_{j=0}^{\nu-1}|\eta|^j\leq \left(1+M_1\cdot\sum_{j=0}^{\nu-1}|\eta|^j\right)\cdot M_h
$$
and the first item of Lemma \ref{lemma_bound_tail} (already proven) implies
$$
|h^{\geq n}(\mu)|\leq  \eta^{-\nu}\cdot|\mu|^\nu\cdot \left(1+M_1\cdot\sum_{j=0}^{\nu-1}|\eta|^j\right)\cdot M_h=|\mu|^\nu\cdot M_2\cdot M_h
$$
with $M_2=\eta^{-\nu}\cdot \left(1+M_1\cdot\sum_{j=0}^{\nu-1}|\eta|^j\right)$ only depending on $\nu$ and $\eta$.
 
\newpage

\nocite{*}
\newpage
\bibliographystyle{plain}
\bibliography{references}

\end{document}